\documentclass{amsart}

\usepackage{graphicx}
\usepackage{latexsym,color}
\usepackage{amssymb}
 \usepackage{amsfonts}
 \usepackage{amsmath}
\usepackage{amsthm}
\usepackage{comment}

\newtheorem{thm}{Theorem}[section]
 \newtheorem{dfn}[thm]{Definition}
 \newtheorem{lem}[thm]{Lemma}
 \newtheorem{rem}[thm]{Remark}
\newtheorem{cor}[thm]{Corollary}

\newtheorem{asm}[thm]{Assumption}
\newtheorem{prp}[thm]{Proposition}
\newtheorem{exm}[thm]{Example}

\def\theequation{\thesection.\arabic{equation}}
\numberwithin{equation}{section}
\newcommand{\cA}{\mathcal{A}}
\newcommand{\cB}{\mathcal{B}}
\newcommand{\cC}{\mathcal{C}}

\newcommand{\cG}{\mathcal{G}}
\newcommand{\cH}{\mathcal{H}}

\newcommand{\cK}{\mathcal{K}}
\newcommand{\cL}{\mathcal{L}}
\newcommand{\cM}{\mathcal{M}}
\newcommand{\cN}{\mathcal{N}}

 \newcommand{\cQ}{\mathcal{Q}}

\newcommand{\cX}{\mathcal{X}}

\newcommand{\cZ}{\mathcal{Z}}

\newcommand{\Om}{{\Omega}}
\newcommand{\om}{{\omega}}

\newcommand{\eps}{{\epsilon}}

\def \Q{\mathbb{Q}}
\def \R{\mathbb{R}}

\def \dP{\mbox{\bf{P}}}
\def \dD{\mbox{\bf{D}}}

\def \fH{\mathfrak{H}}
\def \fK{\mathfrak{K}}
\def \fA{\mathfrak{A}}
\def \fB{\mathfrak{B}}
\def \fC{\mathfrak{C}}
\def \fD{\mathfrak{D}}
\def \fU{\mathfrak{U}}
\def \fL{\mathfrak{L}}
\def \fF{\mathfrak{F}}

\def \fM{\mathfrak{M}}

\def \fS{\mathfrak{S}}
\def \fI{\mathfrak{I}}

\def \fa{\mathfrak{a}}
\def \fb{\mathfrak{b}}
\def \fc{\mathfrak{c}}

\def \ff{\mathfrak{f}}
\def \fn{\mathfrak{n}}
\def \fz{\mathfrak{z}}

\def \fh{\mathfrak{h}}
\def \fg{\mathfrak{g}}

\def \fm{\mathfrak{m}}

\def\1{{\bf{1}}}
\def\0{{\bf{0}}}
\def\e{{\bf{e}}}

\def\p{{\prime}}
\def\pp{{\prime \prime}}
\def\ppp{{\prime \prime\prime}}

\def\reff#1{{\rm(\ref{#1})}}

\begin{document}

\title{Constrained Optimal Transport}
\thanks{ Research is partly supported by the ETH Foundation, 
the Swiss Finance Institute and a Swiss National
Foundation Grant SNF 200021-153555.}
\author{
    Ibrahim Ekren
     \address{ETH Z\"urich, Departement f\"ur Mathematik, R\"amistrasse 101, 
    CH-8092, Z\"urich, Switzerland, email
    \texttt{ibrahim.ekren@math.ethz.ch}.      }
\and H.~Mete Soner
     \address{ETH Z\"urich, Departement f\"ur Mathematik, R\"amistrasse 101, CH-8092, Z\"urich, Switzerland, and Swiss Finance Institute, email \texttt{mete.soner@math.ethz.ch}. }
}
\date{\today}

\begin{abstract}
The classical duality theory
of Kantorovich \cite{Kan} and Kellerer \cite{Kel}
for the classical optimal transport 
is generalized to an  abstract framework
and a characterization 
of the dual elements is provided.
This abstract generalization is set in a
Banach lattice $\cX$ with a order unit.
The primal problem is given
as the supremum over a convex subset of 
the positive unit sphere of
the topological dual of $\cX$
and the dual problem is defined 
on the bi-dual of $\cX$.  
These results are then applied to 
several extensions of the
classical optimal transport.
\end{abstract}
 

\maketitle \markboth{Ekren \& Soner}
{Constrained Optimal Transport}

\renewcommand{\theequation}{\arabic{section}.\arabic{equation}}
\pagenumbering{arabic}

\section{Introduction}
\label{s.introduction}

The Kantorovich relaxation \cite{Kan} of Monge's optimal
transport problem \cite{Mon} is to maximize
$$
\eta(f):= \int_{\R^d\times \R^d} f(x,y)\ \eta(dx,dy),
$$
over all probability measures $\eta$ that have
given marginals $\mu$ and $\nu$. Kantorovich 
proved that the convex dual of this problem
is given by,
$$
\dD_{ot}(f):= \inf\left\{ \mu(h) + \nu(g)\ :\ \
h(x)+g(y) \ge f(x,y)\ \forall (x,y) \in \R^d \times \R^d\ \right\}.
$$
Indeed, a standard application of the Fenchel-Moreau
theorem shows that these two problems have the 
same value when $f$ is continuous and bounded.
We refer the reader to the lecture notes
of Ambrosio \cite{Amb}, the classical books
of Rachev and R\"uschendorf \cite{RR}, Villani \cite{Vil}
and the references therein  
for more information, and to
the recent article of Zaev \cite{Z} which
provides a new approach
to duality.

The above duality can be seen
as a consequence of a pairing between
the  primal measures and a set
of dual functions.  Indeed,
let $\cQ_{ot}$ be the set of
all probability measures with given marginals
$\mu$ and $\nu$ and $\cH_{ot}$ be the set
of all continuous and bounded functions of the form,
$$
k(x,y)= (h(x)-\mu(h)) + (g(y)- \nu(h)), \quad (x,y) \in \R^d \times \R^d,
$$
for some $h, g \in \cC_b(\R^d)$.
Then, we can rewrite the dual problem 
compactly as follows,
$$
\dD_{ot}(f)= \inf\left\{ c\in \R\ :\ \exists k \in \cH_{ot}\ \
{\mbox{such that}}\ \ c+k \ge f
\ \right\}.
$$

\noindent
Moreover, the dual functions $\cH_{ot}$ and 
the primal measures $\cQ_{ot}$ are in 
duality in the sense that $\cQ_{ot}$ is the set
of all probability measures that annihilate $\cH_{ot}$.

We extend the classical optimal transport 
problem based on this duality between the dual and 
primal elements.  Namely, we start with a
with a convex subset $\cQ$ of the positive unit sphere
in the topological dual $\cX^\p$
of a Banach lattice $\cX$.
We assume that there exists a
closed subspace $\cH_\cQ \subset \cX$
such that $\cQ$ is
given as the intersection 
of $\cH_\cQ^\perp$ (the annihilator of $\cH_\cQ$) 
with the positive unit sphere
in the topological dual $\cX^\p$.
Then, a direct application of
the classical Fenchel-Moreau Theorem
yields,
\begin{align}
\label{e.xd}
\dP(f)&:= \sup_{\eta \in \cQ} \eta(f)\\
\nonumber
&=\dD(f):=
\inf\{ c \in \R\ :\ \exists \ h \in \cH_\cQ\ \ 
{\mbox{s.t.}} \ \ c \ \e + h \ge f \},
\quad \forall f \in \cX,
\end{align}
where $\e$ is a unit oder in the Banach lattice $\cX$.
In fact Corollary \ref{c.cont} proves this duality
as a consequence of a general result. 

The above result is proved by applying
convex duality to the map $\dD :\cX \to \R$.
When $\cX = C_b(\Om)$ with a given topological 
space $\Om$, one could also consider
$\dD$ as a Lipschitz, convex function on bounded
and Borel measurable functions, $\cB_b(\Om)$.  
Then, the convex dual of this problem
would be a function on the topological
dual of $\cB_b(\Om)$, namely,
the set of
bounded and finitely additive functionals, $ba(\Om)$.
Hence, to have the duality for all
bounded and measurable functions and not only for 
continuous ones,
one needs to augment the primal measures
by adding an appropriate subset of $ba(\Om)$.
Indeed, Example 8.1 of \cite{BNT} shows 
that this extension to $ba(\Om)$ is necessary 
if one fixes the dual elements $\cH_\cQ$  or equivalently
the dual problem $\dD$.  

Instead, we start with  the primal problem 
defined on the bidual given by,
$$
\dP: \fa \in \cX^\pp \ \ \to \ \ 
\dP(\fa):= \sup_{\eta \in \cQ} \fa(\eta),
$$ 
where $\cQ$ is a given closed convex set in the positive
unit sphere of $\cX^\p$. Since one may view 
$\cX$ as a subset of its bidual, this approach
includes the duality \reff{e.xd}.  Also,
with appropriate choices of $\cX$ one may embed
$\cB_b(\Om)$ as a closed subset of $\cX^\pp$.

Once $\dP$ is defined on a dual space,
the duality can be proved by direct
separation arguments.  In particular,
we prove in Theorem \ref{t.main} that
\begin{align*}
\dP(\fa)&:= \sup_{\eta \in \cQ} \fa(\eta)\\&
=\dD(\fa):=
\inf\{ c \in \R\ :\ \exists\ \fz \in \fK_\cQ\ \ 
{\mbox{s.t.}} \ \ c \ \e + \fz = \fa \},
\quad \forall \fa \in \cX^\pp,
\end{align*}
where the dual set $\fK_\cQ$ is given by,
$$
\fK_\cQ:= \{ \fz \in \cX^\pp\ :\
\fz(\eta) \le 0\ \ \forall \eta \in \cQ\ \}.
$$
Moreover, by its definition $\fK_\cQ$
is a convex cone and is weak$^*$ closed.
Yoshida \cite{Y} shows that such sets in a dual space
are {\em{regularly convex}} as defined by
by Krein \& \v{S}mulian \cite{KS}.  The defining property
of regular convexity is convex separation in the pre-dual 
(see Definition \ref{d.regular} below). 
This allows us to prove the stated duality in $\cX^\pp$
with a fixed primal set $\cQ$ in $\cX^\p$.  In particular,
we identify the dual elements {\em{without}}
augmenting $\cQ$
or equivalently without
extending $\cQ$ to $ba(\Om)$.  Additionally, Theorem \ref{t.main}
proves that on any closed subspace $\fL$
of $\cX^\pp$, the
duality can hold only with the dual set $\fL\cap \fK_\cQ$.

In the applications further characterization
of $\fK_\cQ$ is desired.
Indeed, for the classical 
optimal transport with $\cX=C_b(\Om)$
and $\Om=\R^d \times \R^d$,
Proposition \ref{p.kot} 
proves that the dual set 
$\fK_{ot}:=\fK_{\cQ_{ot}}$
is given by $\fH_{ot}+\cX^\pp_-$,
where $\fH_{ot}=(\cH_{ot}^\perp)^\perp$ is the annihilator of 
 the subspace $\cH_{ot}^\perp$.
The set $\fH_{ot}$ is also characterized as 
the sum of two natural sets.
Then, the duality in  $\cB_b(\Om)$ is proved in
Proposition \ref{p.lsonsuz} as a consequence of 
these results.

Our approach
 is quite different than the previous studies and is based 
on the notion of regular convexity 
as developed by Krein \& 
\v{S}mulian \cite{KS}. 
Indeed, as a general result
proved in Lemma \ref{l.connect},
 the dual set $\fK_\cQ$
is the weak$^*$ closure, or equivalently the
regular convex envelope of $\fH_\cQ+\cX^\pp_-$,
where $\fH_\cQ:= (\cH_\cQ^\perp)^\perp$.
Then, to prove that these two sets are
equal it suffices to show that 
$\fH_\cQ+\cX^\pp_-$ is weak$^*$ closed.
The main difficulty in proving this
or in characterizing  $\fH_\cQ$ 
emanates from the fact that sum of unbounded
regularly convex sets may not be regularly convex.
We use two results from \cite{KS}
to overcome this.  One is the classical
Krein \& 
\v{S}mulian Theorem.
It states  
that a set is regularly convex if 
and only if its intersection with all
bounded balls are  regularly convex.  Secondly,
sums of bounded regularly convex sets 
is again regularly convex.  
Therefore, in applications,
our method necessitates to
prove uniform pointwise
estimates
of the decomposition of dual elements.
Indeed, Lemma \ref{l.technical} proves 
this estimate for the optimal transport
and allows for duality results
Proposition \ref{p.kot} in $\cC_b^\pp(\Om)$
and Proposition \ref{p.lsonsuz} in $\cB_b(\Om)$.
A similar estimate
for super-martingales is obtained in 
Step 4 of the proof of Proposition \ref{p.drc}
and the inequality \reff{e.est} in
Theorem \ref{t.mot}  proves it
for the  martingale optimal transport.

In Section
\ref{s.cot}, we successfully apply this technique to an
extension of the optimal transport which
we call {\em{constrained
optimal transport}}.  
In this problem, the set of primal measures
are further constrained by specifying their actions
on a finite dimensional subset of $\cX$.  This class
of problems was also considered 
by Rachev and R\"uschendorf in  \cite{RR}[Section 4.6.3] for
lower semi-continuous functions. 
Proposition \ref{p.cot}
proves the duality for this extension
by the outlined method in the bidual
and also in $\cB_b(\Om)$.  

A motivating example of the abstract extension is the
martingale optimal transport. In this
problem,
$\cQ_{mot}$ is the set of probability
measures in $\cQ_{ot}$
that also annihilate all functions
of the form $\gamma(x)\cdot (x-y)$. 
Indeed, let $\cH_{mot}$ be the set of 
all linearly growing functions of
the form 
$$
k(x,y)=(h(x)-\mu(h))+(g(y)-\nu(g)) +\gamma(x)\cdot(x-y),
\quad
(x,y)\in \Om,
$$ 
for some linearly growing, continuous functions $h, g$ 
and a bounded, continuous vector valued function  $\gamma$.
Then, $\cQ_{mot}$ is the 
intersection of $\cH_{mot}^\perp$ with the unit
positive sphere.
This problem can also be seen as
a constrained optimal transport
but the dual set is now
enlarged with countably and not finitely many functions.

The martingale optimal transport is first introduced in 
discrete time by Beiglb\"ock, Henry-Labord\`ere and Penkner
\cite{BHLP} and in continuous time
by Galichon, Henry-Labord\`ere and Touzi \cite{GLT}.
The main motivation for this extension
comes from model-free finance or robust hedging
results of Hobson \cite{Hob} and Hobson and Neuberger 
\cite{H5}.  
Initial papers \cite{BHLP,GLT}
also prove the duality for continuous functions.
The duality is then further extended  by Dolinsky and the second author
\cite{DS,DS1,DS2} to 
the case when $\Om$
is the Skorokhod space
of c\`adl\`ag functions by discretization techniques
and later Hou and Ob{\l}{\'o}j \cite{HO} extended these by
considering further constraints.
Also recent manuscripts \cite{GTT1,GTT2}
use the S-topology of Jakubowski
in the Skorokhod space to study
the properties of
the martingale optimal transport.

Several other types of extensions 
of the martingale optimal transport duality are
studied in the literature.
Indeed, especially in financial applications,
it is needed
to relax the pointwise inequalities in the definition
of the dual problem.
The first relaxation is already given in 
the initial paper \cite{GLT} by
using the quasi-sure framework developed
in \cite{STZ,STZ1}. 
In the other types of extensions,
one keeps $\Om=\R^d \times \R^d$ 
but studies the duality for all bounded
measurable  not only
continuous or upper semi-continuous 
functions.
This problem poses interesting
new questions. 
In one space dimension,
they are analyzed thoroughly
in a recent paper by Beiglb\"ock, Nutz and Touzi
\cite{BNT}.  This paper contains, in addition to the
characterization of the dual set, 
several motivating examples and counter-examples.
A recent manuscript \cite{NS} studies 
the super-martingale couplings.

We study the problem of martingale optimal
transport in the 
Banach lattice of linearly growing
continuous functions.  In
Theorem \ref{t.mot}, we
obtain a complete characterization
of the set $\fH_{mot}=(\cH_{mot}^\perp)^\perp$ that annihilates
the primal signed measures in 
$\cH_{mot}^\perp$.
However,
the set $\fH_{mot} +\cX^\pp_-$ is not
equal to $\fK_{mot}=\fK_{\cQ_{mot}}$ as shown
in Example \ref{ex.gap}
and in Example 8.4 of \cite{BNT}.
Similarly, in Section \ref{s.mart},
we study the related problem
defined through martingale measures.
These problems are very closely related
to the convex functions and also to the 
classical result of Strassen \cite{S}.
These connections are made precise in Section
\ref{s.convex}.

In a series of papers, 
Bartl, Cheredito, Kupper
and Tangpi \cite{BCKT,CKT1,CKT}
also develop a functional analytic 
framework for duality problems
of these type.  They, however, start with
the dual problem
and characterize the 
primal measures using
semi-continuity assumptions
on the dual functional. 
In a large variety of problems,
including the financial markets with friction,
they very efficiently  obtain duality results for
upper semi-continuous functions. Another related
subject is the model-free fundamental theorem 
of asset pricing.  In recent years,
many interesting results in this direction 
have been proved \cite{ABS,BN1, BN,  milano, milano1}.
These results essentially start
with the dual elements and define the 
primal measures, $\cQ$ as their annihilators.  Then,
their main concern is to prove that 
$\cQ$ contains elements that 
are countably additive Borel measures.  In this manuscript,
we start with the set $\cQ$ as a subset 
of $\cX^\p$ and prove duality.

The paper is organized as follows.  
Section \ref{s.pre} introduces
the notations and basic results 
used in the paper.
Section \ref{s.setup} defines
the abstract problem
and introduces the regular convexity.
The duality results are proved in
the next section.
Subsection \ref{ss.main} proves the first
duality result in $\cX^\pp$
through $\fK_\cQ$
and the complete 
duality in $\cX$ is established
in the subsection \ref{ss.dualityinX}.  
The necessary and sufficient conditions
for duality with a dual space of
the form $\fH +\cX^\pp_-$ with a subspace $\fH$, 
is obtained in Theorem \ref{t.char} in subsection 
\ref{ss.lower}.
Subsection \ref{ss.factor}
proves a duality result
in the quotient spaces.
The classical optimal transport is
studied in Section \ref{s.ot}
and its extension to
constrained optimal transport in Section \ref{s.cot}.
Section \ref{s.mart} defines
and characterizes martingale measures.
These results are used in Section \ref{s.mot}
to study the multi-dimensional
martingale optimal transport.
Final section states results for the convex functions
defined on the bidual.

\section{Preliminaries}
\label{s.pre}

For convex closed subsets of $X, Y \subset \R^d$, we denote $\Omega:=X \times Y$,
 and for
a Banach lattice $\cX$, we use the 
following standard notations for which
we refer to the 
classical books of Aliprantis \& Border \cite{AB}
and Yoshida \cite{Y} or
to the lecture notes by Kaplan \cite{Kap}.

\begin{itemize}
\item $\cX^\prime$ is the topological dual of $\cX$
and $\cX^\pp$ is its bidual,
\item $\cB_b(\Omega)$ is the Banach space
of bounded real valued  
Borel measurable functions with the supremum norm,
\item $C(\Om)$ is the set of all continuous real-valued  functions,
\item $C_b(\Om)$ is the Banach lattice of all 
continuous real-valued bounded functions
with the supremum norm.
\end{itemize}
    
For $h\in C_b(X)$ and 
$g \in C_b(Y)$, we set
$$
(h \oplus g )(x,y):= h(x)+g(y), \quad
\forall \om=(x,y) \in \Om.
$$
For $\eta \in \left(C_b(\Om)\right)^\p$, its {\em{marginals,}}
$\eta_x \in \left(C_b(X)\right)^\p$ and 
$\eta_y \in \left(C_b(Y)\right)^\p$ are given by
$$
\eta_x(h):= \eta\left(h \oplus{\bf{0}}\right),\quad
\eta_y(g):= \eta\left({\bf{0}}\oplus g\right),
$$
where ${\bf{0}}$ is the constant function
identically equal to zero.

The Banach space $\cX$ is embedded 
into $\cX^{\prime \prime}$ by the canonical mapping,
$$
x \in \cX\ \mapsto \ \fI(x) \in \cX^{\prime\prime}
\quad
{\mbox{where}}
\quad
\fI(x)(x^\prime):= x^\prime(x), \ \forall x^\prime \in \cX^\prime.
$$
 Clearly this notation of $\fI$ depends on the 
 underlying space and we suppress this
 dependence in our notation.  
 
On $\cX^\prime$, we use the order 
induced by the order of $\cX$.
Let $\cX^\prime_+$ be the set of all
positive elements in $\cX^\prime$, i.e.,
$\eta \in \cX^\prime_+$ if
$ \eta(f) \ge 0$ for every
$f \in \cX$ and $f \ge 0$.
We define $\cX^\prime_-$ similarly.
On $\cX^{\prime\prime}$ 
we use the order induced by $\cX^\prime$.

We always assume that the Banach lattice $\cX$
is an $AM$-space endowed with the lattice norm 
induced by an order unit ${\bf{e}} \in \cX_+$, i.e.,
\begin{equation}
\label{e.norm}
\| f \|_\cX = \inf \left\{ c \in \R\ : \
-{\bf{c}} \le f \le {\bf{c}} \ \right\}, \quad
{\mbox{where}} \quad 
{\bf{c}}:= c\ {\bf{e}}.
\end{equation}
Then, the bidual $\cX^\pp$ is also
an $AM$-space with $\fI({\bf{e}})$
as its order unit; \cite{AB}[Theorem 9.31].  
Moreover,
\begin{equation}
\label{e.dnorm}
\| \fa \|_{\cX^\pp} = \inf \left\{ c \in \R\ : \
-{\bf{c}} \le f \le {\bf{c}} \ \right\}, \quad
{\mbox{where}} \quad 
{\bf{c}}:= c\ \fI({\bf{e}}).
\end{equation}
We also have,
\begin{equation}
\label{e.norm2}
 \eta({\bf{e}}) = \|\eta^+\|_{\cX^\prime} -\|\eta^-\|_{\cX^\prime},
 \quad \forall \ \eta \in \cX^\p.
\end{equation}

We denote the
unit ball in $\cX$ by 
$B_1$  and set $B_+:= B_1 \cap \cX_+$.
Similarly, we let $B_1^\prime$ and $B_1^{\prime\prime}$
be the unit balls
in $\cX^\prime$ and  in $\cX^{\prime\prime}$, respectively.
We set $B_+^{\prime} := B_1^\prime \cap \cX_+$,
$B_+^{\prime\prime} := B_1^{\prime\prime} \cap \cX_+$.

For a given 
a subset $A$ of a Banach space $\cZ$
and a subset $\Theta \subset \cZ^\prime$,
the annihilator of $A$ and the pre-annihilator 
of $\Theta$ are given by, 
$$A^\perp := \left\{\eta \in \cZ^{\p} \ :\
\eta(a)=0, \ \  \forall a \in A\ \right\},\quad
\Theta_\perp := \left\{a \in \cZ \ :\
\eta(a)=0, \ \  \forall\eta \in \Theta\ \right\}.
$$
It is clear that $A^\perp$ is weak$^*$
closed in $\cZ^\p$ and $\Theta_\perp$
is weakly closed in $\cZ$.

For a scalar $c$, $\bf{c}$  denotes the 
function  equal to $c\ {\bf{e}}$.
Clearly, this notation depends on the order unit
but with an abuse of notation, we use the same notation
in all domains.

Throughout the paper,  we mostly use
the Banach lattice $C_b(\Om)$  or 
the space $C_\ell(\Om)$ of linearly
growing continuous functions defined by,
$$
C_\ell(\Om):= \left\{ \ f \in C(\Om)\ :\
\| f\|_\ell < \infty\ \right\},
$$
where, with  $\ell_X(x):= 1+|x|$,  $\ell_Y(y):= 1+|y|$
and $\ell(x,y)= \ell_X(x)+ \ell_Y(y)$, 
\begin{equation}
\label{e.lnorm}
\| f\|_\ell := \sup_{ \om \in \Om} \ \frac{|f(\om)|}{\ell(\om)}.
\end{equation}
To simplify the presentation, we denote
$$
\cC_b:=C_b(\Om) \quad
{\mbox{and}}
\quad
\cC_\ell:=C_\ell(\Om).
$$
It is clear that both of these spaces are $AM$-spaces 
and the order unit in $\cC_b$ is ${\bf{e}} \equiv 1$.  The space $\cC_\ell$
has the order unit ${\bf{e}}(x,y)= \ell(x,y)$.  Moreover, 
the weighted spaces $\cC_{\ell_X}(X)$ and
$\cC_{\ell_Y}(Y)$ are also $AM$-spaces
with order units $\ell_X$ and $\ell_Y$, respectively.

We also use the notation
$\cC:= C(\Omega)$
and view it as a Frechet space.
Then,
$\cC^\p$ is equal to $ca_{r,c}(\Om)$,
all countably additive, regular measures
that are compactly supported. 
\vspace{2pt}

\section{Abstract Problem}
\label{s.setup}
Let $\cX$ be a Banach lattice with
 a lattice norm
given by \reff{e.norm} and with an
order unit ${\bf{e}}$.  Recall the notation,
${\bf{c}} := c \ {\bf{e}}$ and with an abuse
of notation, we use the same notation in the bidual as well.
Let $\partial B^\p_1$ be the unit sphere in $\cX^\p$.  In view of
\reff{e.norm2}, we have
$$
\partial B^\p_+ := \partial  B^\p_1 \cap \cX^\p_+
= \left\{ \eta \in \cX^\p_+\ :\ 
\eta({\bf{e}})=1\ \right\}.
$$
The starting point of our analysis is a 
closed convex set $\cQ \subset \partial B^\p_+$.
We make the following standing assumption.
\begin{asm}
\label{a.main}
We assume that $\cQ$
is a non-empty, closed, convex 
subset of $\cX^\p$
and that there exists a closed subspace $\cH_{\cQ}\subset \cX$ 
such that 
$$
\cQ=\cH_\cQ^\perp \ \cap  \partial B^\p_+.
$$ 
\end{asm}
Set
$$
\cA_\cQ:= \cH_\cQ^\perp,
\qquad
\cC_\cQ:= \left\{ \lambda \eta\ :\ \eta \in \cQ,\ \ \lambda \ge 0\ 
\right\} = \cH_\cQ^\perp \cap \cX^\p_+.
$$
Note that the closed linear span of $\cQ$
is equal to  $\cC_\cQ-\cC_\cQ$
and is always a subset of $\cA_\cQ$.
But in general
this inclusion could be strict.
\vspace{2pt}

\subsection{Definitions}
\label{ss.def}
Given $\cQ$, 
the {\em{constrained optimal transport}} is given by,
$$
\dP(\fa; \cQ):= \sup_{\eta \in \cQ}\ \fa(\eta),
\quad
\fa \in \cX^{\prime\prime}.
$$

On the dual side,
we start with a cone $\fK \subset \cX^\pp$
satisfying
\begin{equation}
\label{e.cond}
\fz \in \fK\ \  {\mbox{and}} \ \ 
\fn \in \cX^{\prime\prime}_-\ \ 
\Rightarrow
\quad 
\fz +\fn \in \fK .
\end{equation}
Then, the {\em{dual constrained optimal
transport}} problem is defined by,
$$
\dD(\fa; \fK):= \inf \left\{ \ c \in  \R\ : \
\exists\ \fz \in \fK\ \ 
{\mbox{such that}}\ \ 
 {\bf{c}}+ \fz = \fa\ \right\},\quad
\fa \in \cX^{\prime\prime}.
$$
We always use the convention that
the infimum over an empty set is plus infinity.

The chief concern of this paper
is to relate these two problems.
In particular, the following sets
are relevant,
\begin{align}\label{eq:defh}\fH_\cQ:= \cA_\cQ^\perp, \qquad
\fK_{\cQ}:= \left\{ \fz \in \cX^{\prime\prime}\ :\
\fz(\eta) \le 0, \ \ \forall \  \eta \in \cQ\ \right\}.
\end{align}
It is then immediate that
$\fK_{\cQ}= \left\{ \fz \in \cX^{\prime\prime}\ :\
\fz(\eta) \le 0, \ \ \forall \  \eta \in \cC_\cQ\ \right\}.$

\begin{rem}
\label{r.kh}
{\rm{These two sets are closely related
to each other as we always have the inclusion,
$\fH_\cQ + \cX^{\prime\prime}_- \subset \fK_{\cQ}$.
We also show in Lemma \ref{l.connect}
that the weak$^*$ closure of $\fH_\cQ + \cX^{\prime\prime}_-$
is equal to $\fK_\cQ$.
On the other hand
this inclusion might be strict
as shown in Example \ref{ex.gap} below.
Indeed, $\fK_{\cQ}$
is always weak$^*$ closed, while $\fH_\cQ + \cX^{\pp}_- $
may not even be strongly closed.  
\qed
}}
\end{rem}

\subsection{Properties of $\dD$}
\label{ss.properties}

We prove several easy properties of $\dD$
for future reference.  Let $\overline \fK$ be the closure of
$\fK$ under the strong topology of $\cX^{\prime\prime}$.

\begin{lem}
\label{l.properties} Suppose that
$\fK$ satisfies \reff{e.cond}.  Then,
for every $\fa \in \cX^\pp$,
$$
\dD(\fa;\fK) =\dD(\fa;\overline{\fK})\le \|\fa\|_{\cX^\pp}.
$$
Moreover, if $\dD(\fa;\fK)>-\infty$,
then there exists $\fz_\fa \in \overline{\fK}$
satisfying,
$$
\fa= \fz_\fa + \dD(\fa;\overline{\fK})\ \e.
$$
\end{lem}

\begin{proof}  Fix $\fa \in \cX^\pp$.
By \reff{e.dnorm},
$\|\fa\|_{\cX^\pp}\ \e \ge \fa$.  Since $\cX^\pp_-\subset \fK$,
$\dD(\fa;\fK) \le \|\fa\|_{\cX^\pp}$. 

Let $(c,\fz)\in \R\times \overline{\fK}$ be such that
$\fa=c\ \e+ \fz$.  Let $\{\fz_n\}_n \subset \fK$ be 
a sequence that converges to $\fz$.
In view of \reff{e.dnorm}, 
$$
\fa= c\  \e + \fz \le \left[c + \|\fz-\fz_n\|_{\cX^\pp}\right]\e
+ \fz_n.
$$
Hence, by \reff{e.cond},
$$
\fn_n:= \fa -\left[c + \|\fz-\fz_n\|_{\cX^\pp}\right]\e- \fz_n \in \cX^\pp_-
\quad
\Rightarrow
\quad
\fz_n+\fn_n \in \fK.
$$
Moreover, 
$$
\fa=  \left[c + \|\fz-\fz_n\|_{\cX^\pp}\right]\e
+\left[ \fz_n +\fn_n\right]
\quad
\Rightarrow
\quad
\dD(\fa;\fK) \le c + \|\fz-\fz_n\|_{\cX^\pp}.
$$
Since above holds for every pair $(c,\fz) \in \R\times \overline{\fK}$
satisfying $\fa=c \ \e + \fz$, we conclude that
 $\dD(\fa;\fK) \le \dD(\fa;\overline{\fK})$.  The opposite
 inequality is immediate, since $\fK \subset \overline{\fK}$.
 
Suppose that $\dD(\fa;\fK)>-\infty$.
Then, there is a sequence $(c_n,\fz_n)\in \R \times \fK$
such that $\fa= c_n \e +\fz_n$ and
$c_n$ tends to $\dD(\fa;\fK)$
as $n$ tends to infinity.  
Then, 
$$
\|\fz_n-\fz_m\|_{\cX^\pp} = |c_n-c_m|, \quad
\forall \ n, m.
$$
Hence, $\{\fz_n\}_n$ is a Cauchy sequence.
Let $\fz_\fa \in \overline{\fK}$ be its limit point.
Then, 
$$
\fa= \fz_\fa + \dD(\fa;\overline{\fK})\ \e.
$$
\end{proof}
\vspace{2pt}

\subsection{Regular Convexity}
\label{ss.regular}

The notion of regular convexity defined in  \cite{KS}
by Krein \& \v{S}mulian is useful
in this context as it allows for convex separation   
in the pre-dual. 

\begin{dfn}[Regular Convexity]
\label{d.regular} Let $\cZ$ be a Banach space.
A subset $\fA \subset \cZ^\prime$ is called
regularly convex if for any $\fb \not \in \fA$,
there exists $\eta \in \cZ$ such that
$$
\sup_{\ff \in \fA}\ \ff(\eta) < \fb(\eta).
$$
\end{dfn}
It is proved by Yoshida \cite{Y} that a set is regularly
convex if and only of it is convex and is weak$^*$ closed.
We also recall a condition for regular convexity
in Appendix at section \ref{ap.rc}. 

The sets $\fH_\cQ$ and $\fK_{\cQ}$ defined earlier  
are both weak$^*$ closed and convex.  Hence,
they are regularly convex. Moreover,
due to general facts of functional analysis and the 
relation between $\cA_\cQ$ and $\cQ$, 
we have the following connection
between  the spaces
$\fK_{\cQ}$
 and  $\fH_\cQ +\cX^{\prime\prime}_-$.

\begin{lem}
\label{l.connect} Under the Assumption \ref{a.main},
the weak$^*$ closure
of  $\fI(\cH_\cQ) +\cX^{\prime\prime}_-$
and $\fH_\cQ +\cX^{\prime\prime}_-$
are equal to $\fK_{\cQ}$.
\end{lem}
\begin{proof}
Let $\fK$ be the weak$^*$ closures of 
$\fI(\cH_\cQ) +\cX^{\prime\prime}_-$.
Since $\fI(\cH_\cQ) +\cX^{\prime\prime}_- \subset \fK_{\cQ}$,
we also have
$\fK \subset \fK_{\cQ}$.
For a contraposition 
argument, assume that there exists 
$\fa_0 \in \fK_{\cQ} \setminus \fK$.
Since $\fK$ is convex and weak$^*$ closed,
it is regularly convex.  Hence, there exists
$\eta_0 \in \cX^\prime$ satisfying,
\begin{equation}
\label{e.a0}
c_0:= \sup_{\fz \in \fK}\fz(\eta_0) < \fa_0(\eta_0).
\end{equation}
Since $ \fK$ is a cone and contains $0$,
we conclude that $c_0=0$. 
Therefore, $\fI(h)(\eta_0) \le 0$
for every $h  \in \cH_\cQ$.
By the fact $\cH_\cQ$ is linear,
we conclude that
$\eta_0 \in \cH_\cQ^\perp$.
Also, since $\cX^{\prime\prime}_-\subset \fK$, $\eta_0 \ge0$.  
In particular, $\eta_0 \in \cH_\cQ^\perp \cap \cX^\prime_+$ and by 
Assumption \ref{a.main}, and the 
definition $\cC_\cQ$,
we conclude that $\eta_0 \in \cC_\cQ$.
On the other hand, $\fa_0 \in \fK_{\cQ}$.
Hence, $\fa_0(\eta_0) \le 0$.  This
is in contradiction with \reff{e.a0} and the fact that
$c_0=0$.
This proves that the
weak$^*$ closure
of $\fI(\cH_\cQ) +\cX^{\prime\prime}_-$ is equal
to $\fK_\cQ$.  

Finally, since $\fH_\cQ + \cX^\pp_- \subset \fK_\cQ$,
and since $\fI(\cH_\cQ) \subset (\cH_\cQ^\perp)^\perp=\fH_\cQ$, 
we have
$$
\fI(\cH_\cQ) + \cX^\pp_-
\subset \fH_\cQ + \cX^\pp_- \subset \fK_\cQ.
$$
We have already shown that
 the weak$^*$ closure
of the smallest set above is 
equal to $\fK_\cQ$.
Hence,
 the weak$^*$ closures
are all above sets are equal to $\fK_\cQ$.
\end{proof}
\vspace{2pt}

\section{Duality}
\label{s.dual}
In this section, we prove several
duality results and also necessary and sufficient
conditions for certain types of duality.

\subsection{Main Duality}
\label{ss.main}

The following is the main duality result.

\begin{thm}[{\bf{Duality}}]
\label{t.main}  

Suppose that $\cQ$ is a convex subset 
of $\cX^\p$ satisfying Assumption \ref{a.main}.
Let $\fL$ be a strongly  closed subspace of $\cX^\pp$
containing $\e$
and $\fK \subset \fL$ be a
cone  satisfying
$$
\fz \in \fK\ \  {\mbox{and}} \ \ 
\fn \in \fL \cap \cX^{\prime\prime}_- \ \
\Rightarrow
\quad  \fz +\fn \in \fK .
$$
Then, the duality
\begin{equation}
\label{e.dual}
\dP(\fa;\cQ)= \dD(\fa;\fK),\quad
\forall \ \fa \in \fL,
\end{equation}
holds if and only if
the strong closure $\overline{\fK}$
of $\fK$ is equal to $\fK_\cQ \cap \fL$.  Moreover,
there is dual attainment in $\overline \fK$.   Namely, for every
$\fa \in \fL$ there exists $\fz_\fa \in \overline \fK$ satisfying,
\begin{equation}
\label{e.attainment}
\dD(\fa;\fK)\ \e + \fz_\fa = \fa.
\end{equation}
\end{thm}

\begin{proof}  
Fix $\fa \in \fL$.  Set $\fL_\cQ:= \fL \cap \fK_{\cQ}$ and
$$
\fD_\cQ := \{ (c,\fz) \in \R \times \fL_\cQ\ :\  {\bf{c}} +\fz = \fa\ \}.
$$
In view of \reff{e.dnorm}, 
$$
\fn: =- \| \fa\|_{\cX^\pp} \ {\bf{e}} + \fa \le 0 .
$$
Hence, trivially, $\fn \in \fL_\cQ$.  Therefore,
$(\| \fa\|_{\cX^\pp}, \fn) \in \fD_\cQ$.  In particular,
$\fD_\cQ$ is non-empty.
Suppose that $(c,\fz) \in \fD_\cQ$.
Let $\eta \in \cQ$.
Then, $\eta({\bf{e}})=1$, $\eta \ge 0$ and
$\fz(\eta)\le0$ for every 
$\fz \in \fL_{\cQ}$.
Consequently,
$$
\fa(\eta)= \eta({\bf{c}})+ \fz(\eta) \le c .
$$
This proves that $\dP(\fa;\cQ) \le \dD(\fa;\fL_{\cQ})$.

Since $\cQ$ is non-empty and $\eta({\bf{e}})=1$ for every
$\eta \in \cQ$, we also conclude that
$$
- \| \fa\|_{\cX^\pp} \le \dP(\fa;\cQ) \le \dD(\fa;\fL_{\cQ}).
$$
Hence, $\dD(\fa;\fL_{\cQ})$ is finite and there exists 
a sequence $(c_n, \fz_n) \in \fD_\cQ$
so that $c_n$ 
monotonically converges to $\dD(\fa;\fL_{\cQ})$.
Then,
$$
\left\| \fz_n-\fz_m \right\|_{\cX^\pp} = |c_n-c_m|
$$
for each $n,m$.  This implies that the strong limit $\fz_\fa$
of the sequence $\fz_n$ exists and satisfies \reff{e.attainment}.
It is clear that $\fK_\cQ$ is closed in the weak$^*$ topology and
hence is also closed in the strong topology.  
Since $\fL$ is closed by hypothesis, $\fz_\fa \in \fL_\cQ= \fL \cap \fK_\cQ$.
Set
$$
c^*:= \dP(\fz_\fa; \cQ).
$$
Since $\fz_\fa \in \fL_\cQ$, $c^* \le 0$.  Moreover,
$\fz_\fa - {\bf{c^*}} \in \fL_\cQ$ and
$$
\left[ \dD(\fa;\fL_{\cQ})+c^* \right] {\bf{e}} +\left[ \fz_\fa - {\bf{c^*}}\right] = \fa
\quad
\Rightarrow
\quad
\left(\dD(\fa;\fL_{\cQ})+c^*,  \fz_\fa - {\bf{c^*}}\right) \in \fD_\cQ.
$$
Since $\dD(\fa;\fL_{\cQ})$ is the minimum over all constants $c$
so that there is $\fz \in \fL_\cQ$ satisfying $(c,\fz) \in \fD_\cQ$
and since $c^* \le 0$,
we conclude that $c^*=0$.  Then,
$\fa= \fz_\fa + \dD(\fa;\fL_{\cQ})\ {\bf{e}}$ and
consequently,
$$
\dP(\fa;\fL_{\cQ}) = \sup_{\eta \in \cQ} \ 
\left[ \fz_\fa + \dD(\fa;\fL_{\cQ})\ {\bf{e}}\right](\eta)
= \dP(\fz_\fa;\fL_{\cQ}) + \dD(\fa;\fL_{\cQ}) = \dD(\fa;\fL_{\cQ}).
$$
Hence the duality on $\fL$ holds when $\fK= \fL_\cQ$.

We continue by proving the opposite implication.
Suppose that the duality \reff{e.dual} holds
for every $\fa \in \fL$. 
Set $\overline{\fK}$ be the strong closure of $\fK$.
Then, by Lemma \ref{l.properties}, $\dD(\cdot;\fK)= \dD(\cdot;\overline{\fK})$.
We first claim
that $\overline{\fK}$ is contained in $\fL_{\cQ}$. 
Fix $\fz_0 \in \overline{\fK}$.
By the
definition of the dual problem, 
$\dD(\fz_0;\overline{\fK})\le 0$.  
Since the duality holds,
$$
\sup_{\eta \in \cQ} \fz_0(\eta) = \dP(\fz_0;\cQ)
=\dD(\fz_0;\fK)= \dD(\fz_0;\overline{\fK})  \le 0.
$$
Hence, $\fz_0 \in \fK_{\cQ}$.
\vspace{4pt}

To prove the 
opposite inclusion, let $\fz^* \in \fL_\cQ$.
Then,
$\fz^*(\eta)\le0$ for every $\eta \in \cQ$.
Since, by hypothesis, the duality holds,
we conclude that
$$
c_0:=\dD(\fz^*;\fK)= \dP(\fz^*;\cQ) 
= \sup_{\eta \in \cQ} \fz^*(\eta) \le 0.
$$
By Lemma \ref{l.properties}, there are $\fa^* \in \overline{\fK}$
such that
$\fz^*= c_0 \ \e + \fa^*$.  Since $c_0 \le 0$, we  have $c_0 \ \e \in \cX^\pp_-$.
Since $\overline{\fK}$ satisfies  \reff{e.cond}, 
we conclude that $\fz^* \in \overline{\fK}$ and
consequently, $\fL_{\cQ} \subset \overline \fK$.
Therefore, 
$\overline \fK =\fL_{\cQ}$ whenever the duality holds.
 \end{proof}
\vspace{2pt}

\subsection{Duality in $\cX$}
\label{ss.dualityinX}

We continue by proving the duality in $\cX$.
For the optimal transport and 
its several extensions, Zaev \cite{Z} also
provides a proof
of this duality when $\cX=C_b(\Om)$.

For any $f \in \cX$,  with an abuse of notation, we write $\dP(f;\cQ)$
instead of $\dP(\fI(f);\cQ)$ and $\dD(f; \cH_\cQ+\cX_-)$ instead
of $\dD(\fI(f);\fI(\cH_\cQ+\cX_-))$.  
Next, we use Theorem \ref{t.main} with $\fL=\fI(\cX)$
to prove duality in $\cX$. This result
can also be proved as a direct
consequence of Theorem 7.51 of  \cite{AB}.

Recall that the sets
$\cH_\cQ$ and $\cC_\cQ$
are defined in Section \ref{s.setup}.

\begin{cor}[{\bf{Duality in $\cX$}}]
\label{c.cont}  Under Assumption \reff{a.main},
$$
\dP(f;\cQ)= \dD(f;{\cH_\cQ+\cX_-}), \quad \forall f \in \cX.
$$
\end{cor}
\begin{proof} Let $\cK$ 
be the strong closure of $\cH_\cQ+\cX_-$. 
Since $\fI(\cX)$ is a closed set, in view of Theorem
\ref{t.main}, it suffices to show that $\cK$ is equal to
$$
\cK_\cQ:=\left\{
k \in \cX\ :\
 \fI(k) \in \fK_\cQ\ \right\}.
$$
Towards a contraposition, assume that there is 
$f_0 \in \cK_\cQ \setminus
\cK$.
Since $\cK$ is closed by its definition, by Hahn-Banach, there exists
$\eta_0 \in \cX^\p$ satisfying,
$$
c_0:=\sup_{f \in \cK} \eta_0(f) < \eta_0(f_0).
$$
Since $\cK$ contains 
$\cH_\cQ+\cX_-$, $c_0=0$
and also $\eta_0 \ge 0$.
Consequently, $\eta_0 \in \cH_\cQ^\perp \cap \cX^\p_+$
and this set is equal to  $\cC_\cQ$.  

Since $\fI(f_0) \in \fK_\cQ$ and $\eta_0 \in \cC_\cQ$,
$\eta_0(f_0)=\fI(f_0)(\eta_0) \le 0$.
This contradicts
the contraposition hypothesis.  
This proves that $\cK_\cQ=\cK=\overline{\cH_\cQ+\cX_-}$
and consequently,
$$
\dP(f;\cQ)=\dD(f;\overline{\cH_\cQ+\cX_-}).
$$
We now conclude by using Lemma \ref{l.properties}.
\end{proof}
\vspace{2pt}

\begin{rem}
\label{r.noattainment}
{\rm{It is clear that the dual attainment
is equivalent to the closedness of the set $\cH_\cQ+\cX_-$.
However, in general, this set 
may not be closed.  In such situations,
 the duality holds without dual attainment.
 
The corollary above shows why the duality is usually
easier to prove in $\cX$. Indeed, as shown in Lemma \ref{l.properties}, 
thanks to the {\em{a priori}} regularity of the value of the dual problem 
with respect to the lattice norm, a hedging set and 
its strong closure gives the same value for the dual problem. 
This invariance with respect to strong closure is 
exactly the crucial ingredient
used above to prove the duality in $\cX$. 
 }}
\qed
\end{rem}
\vspace{2pt}

An alternate approach to duality in $\cX$ is developed
in a  series of papers \cite{BCKT,CKT1,CKT}.  Indeed,
these papers also establish duality for continuous
functions very efficiently for a very general class.  Then, they extend their
results to upper semi-continuous
functions by analytic approximation techniques.

\subsection{Duality with Lower Subspaces}
\label{ss.lower}

We call a set in $\cX^\pp$ a {\em{lower subspace}} if it is
of the form $\fH+\cX^\pp_-$ for some subspace $\fH$.
In this section, we investigate when the duality holds
with these types of dual sets.  

Recall that $\cA_\cQ$, $\cC_\cQ$, $\fK_\cQ, \fH_\cQ$
are defined in Section \ref{s.setup}.
Futher let $\hat \cA_\cQ$ be the linear span of $\cQ$.
Then, $\hat \cA_\cQ= \cC_\cQ-\cC_\cQ$.  Set
$$
\hat \fh_\cQ:= \hat \cA_\cQ^\perp,\qquad
\fA_\cQ:= 
\overline{\hat \fH_\cQ+\cX^\pp_-}.
$$
For any a set $B$ in the dual  of a Banach space,  
$\overline{B}^*$ is the weak$^*$ closure of $B$.
\begin{thm}
\label{t.char} Under the Assumption \ref{a.main},
the following are equivalent:
\begin{enumerate}
\item 
\label{fH}
There exists a subspace $\fH$  of $\cX^\pp$ such that
the duality  on $\cX^\pp$ holds with $\fH +\cX^\pp_-$, i.e.,
$$
\dP(\fa;\cQ)= \dD(\fa; \fH+\cX^\pp_-),
\quad \forall \fa \in \cX^\pp.
$$
\item
\label{fK}
$\fA_\cQ= \fK_\cQ$.
\item
\label{fA}
The duality with $\fA_\cQ$ holds on $\cX^\pp$, i.e.,
$$
\dP(\fa;\cQ)= \dD(\fa; \fA_\cQ),
\quad \forall \fa \in \cX^\pp.
$$
\item 
\label{ppp}
$\overline{\fI(\hat \cA_\cQ)}^{\ *}
 \cap \cX^\ppp_+ = \overline{ \fI(\cC_\cQ)}^{\ *}$.
\end{enumerate}
Moreover, when \reff{fH} holds,
then, $\fH$
is a subset of  $\hat \fH_\cQ$ and 
the strong closure of 
$\fH +\cX^\pp_-$ is equal to $\fK_\cQ$.

\end{thm}

\begin{proof}

$\reff{fH}  \Rightarrow  \reff{fK}$.
For any $\fa \in \fH$ and $\eta \in \cQ$,
$$
\fa(\eta) \le \dP(\fa;\cQ) = \dD(\fa; \fH+ \cX^\pp_-) \le 0.
$$
Since $\fH$ is a subspace, the above implies that
$\fH \subset \hat \fH_\cQ$.  Let $\fK$ be 
the strong closure of $\fH+\cX^\pp_-$.  
Then,  $\fK \subset \fA_\cQ \subset \fK_\cQ$.  

Let $\fa \in \fK_\cQ$.  Then, by the definition of $\fK_\cQ$,
$\dP(\fa;\cQ) \le 0$.  Hence,
$$
\dD(\fa;\fK)= \dP(\fa;\cQ) \le 0.
$$
Since $\fK$ is closed and has the property \reff{e.cond},
there is $\fz_0 \in \fK$ so that
$$
\dP(\fa;\cQ) \e + \fz_0 =\fa.
$$
We use the property \reff{e.cond} once more to conclude 
that $\fa \in \fK$.  So we have proved that
$\fK= \fK_\cQ$.  Since  $\fK \subset \fA_\cQ \subset \fK_\cQ$,
this also proves that $\fA_\cQ=\fK_\cQ$.
\vspace{2pt}

$\reff{fK} \Rightarrow \reff{fA} $. This follows directly
from Theorem \ref{t.main}.
\vspace{2pt}

$\reff{fA} \Rightarrow \reff{fH}$.  In view Lemma 
\ref{l.properties},  the 
primal problem with $\hat \fH_\cQ+\cX^\pp_-$ and
with its closure $\fA_\cQ$ have the same value.
\vspace{2pt}

$\reff{ppp} \Rightarrow \reff{fK}$.  
Towards a contraposition,
suppose that $\fA_\cQ$ is not equal 
to $\fK_\cQ$. Let
$\fa_0 \in \fK_\cQ \setminus \fA_\cQ$.  By Hahn-Banach,
there exists $\aleph_0 \in \cX^\ppp$ such that,
$$
c_0= \sup_{\fa \in \fA_\cQ} \aleph_0(\fa) < \aleph_0(\fa_0).
$$
As it is argued before in similar situations, 
we conclude that 
$\aleph_0 \in \hat \fH_\cQ^\perp \cap \cX^\ppp_+$
and  $c_0=0$.
Since $\hat \fH_\cQ^\perp$ is equal 
to the weak$^*$  closure of $\fI(\hat \cA_\cQ)$,
$\aleph_0 \in \overline{\fI(\hat \cA_\cQ)}^{\ *} \cap \cX^\ppp_+$
and hence, $\aleph_0
\in \overline{\fI(\cC_\cQ)}^{\ *}$.  
It is clear that
$$
\aleph(\fz) \le 0, \quad \forall \fz \in \fK_\cQ,
\ \ \aleph  \in \overline{\cC_\cQ}^{\ *}.
$$
Hence, $\aleph_0(\fa_0) \le 0$.
This contradicts with the contraposition hypothesis.
Hence, $\fA_\cQ=\fK_\cQ$.

\vspace{2pt}

$\reff{fK} \Rightarrow \reff{ppp}$. Suppose that the 
contrary of \reff{ppp} holds.  Then, there is 
$$
\gimel_0 \in \overline{\fI(\hat \cA_\cQ)}^{\ *} \cap \cX^\ppp_+ 
\setminus \overline{ \fI(\cC_\cQ)}^{\ *}.
$$
Since $\overline{ \fI(\cC_\cQ)}^{\ *}$ is regularly convex,
there exists $\fa_0 \in \cX^\pp$ satisfying,
$$
c_0=\sup_{ \gimel \in \overline{ \fI(\cC_\cQ)}^{\ *} }
\gimel(\fa_0) < \gimel_0(\fa_0).
$$
Then, it is clear that $c_0=0$
and consequently, $\fa_0 \in \fK_\cQ$. 
Then, by \reff{fK}, $\fa_0\in \fA_\cQ$ and there exists a sequence
$\fa_n = \fh_n + \fz_n \in \hat  \fH_\cQ+ \cX^\pp_-$ 
converging to $\fa_0$ in the strong topology of
$\cX^\pp$. For each $n$, since $\fh_n \in \hat \fH_\cQ$,
$\gimel_0(\fh_n)=0$ and since $\fz_n \le0$, $\gimel_0(\fz_n) \le0$.
Therefore, $\gimel_0(\fa_n) \le 0$ for each $n$ and 
by letting $n$ tend to infinity, we conclude that
$\gimel_0(\fa_0) \le0$.
This contradicts with the choice of $\gimel_0$,
namely, $\gimel_0(\fa_0)>c_0=0$.
\end{proof}
\vspace{2pt}

\begin{rem}
\label{r.condition}
{\rm{  One could prove the implication
$\reff{ppp} \Rightarrow \reff{fH}$ 
by considering the duality in the space
$\cX^\pp$ and then
applying Corollary \ref{c.cont}.  However,
for structural reasons,
this approach requires
the condition \reff{ppp}.

Example \ref{ex.gap} below shows that, in general
 neither $\fH_\cQ +\cX^\pp_-$ 
nor  $\hat \fH_\cQ +\cX^\pp_-$ are
equal to $\fK_\cQ$.
 Example 8.4 of \cite{BNT} also
demonstrates a similar phenomenon.
\qed
}}
\end{rem}
\vspace{2pt}

\subsection{Factor Spaces}
\label{ss.factor}

In our context, Theorem 12 of \cite{KS}
states that if $\fH$ is a regularly convex subspace of $\cX^\pp$,
then the dual of the subspace $\fH_\perp$
is equal to the quotient space $\cX^\pp/\fH$.
This result provides a statement quite similar
to the duality proved earlier
but with two-sided inequalities.
  
\begin{lem}
\label{l.factor}
Suppose that Assumption \ref{a.main} holds.
Then, $\cX^\pp/\fH_\cQ$
is the topological dual of $\cA_\cQ$.
Consequently,
$$
\sup_{\eta \in \cA_\cQ} \ \frac{\left| \fa(\eta) \right|}
{\|\eta\|_{\cX^\prime}} = \inf_{\fh \in \fH_\cQ}\ \|\fa - \fh\|_{\cX^{\prime\prime}},
\quad \fa \in \cX^{\prime\prime}.
$$
\end{lem}
\begin{proof}
Since by its definition $\cA_\cQ$ is closed,
by the Lemma on page 573 in \cite{KS},
we conclude that
$$
\left(\fH_\cQ\right)_\perp = 
\left(\cA_\cQ^\perp\right)_\perp = \cA_\cQ.
$$
Hence, Theorem 12 of \cite{KS} implies the 
duality statement of the lemma.  
Also observe that for any $\fa \in \cX^{\prime\prime}$,
$$
\sup_{\eta \in \cA_\cQ} \ \frac{\left| \fa(\eta) \right|}
{\|\eta\|_{\cX^{\prime}}} 
= \|\fa\|_{(\cA_\cQ)^\prime} = 
\|\fa\|_{\cX^{\prime\prime}/\fH_\cQ} = 
\ \inf_{\fh \in \fH_\cQ}\ \|\fa - \fh\|_{\cX^{\prime\prime}}.
$$
\end{proof}
  
\begin{rem}
\label{r.analog}  
{\rm{One may interpret the left hand side of 
the above equation
as a primal transport problem
and the right hand side as its dual.
Indeed,  \reff{e.dnorm}
implies  the following duality
with $\tilde \cQ:= \cH_\cQ^\perp \cap B^\pp_1$,
\begin{align*}
\tilde \dP(\fa;\tilde \cQ)&:= \sup_{\eta \in \tilde \cQ} \fa(\eta)\\
&=\tilde \dD(\fa;\fH)
 := \inf \left\{ \ c \ge 0\ :\ 
 \  \exists \ \fh \in \fH \  \
 {\mbox{such that}}\ \ - {\bf{c}} \le \fa -\fh \le  {\bf{c}}\ \right\}.
\end{align*}
Notice that $\tilde \cQ$ is not a subset of 
$\cX^\p_+$ and this is a crucial difference
between the above identity and the duality \reff{e.dual}.
\qed}}
\end{rem}
\vspace{2pt}
 \section{Classical optimal  transport}
 \label{s.ot}
 This section studies the classical
 duality result of Kantorovich \cite{Kan}
 in this context.
The optimal transport duality for general Borel measurable functions
was proved by Kellerer \cite{Kel}
and a very general extension was recently
given by Beiglb\"ock, Leonard and Schachermayer
\cite{BLS}.
We also refer to the lecture notes
of Ambrosio \cite{Amb} and the classical books of 
Rachev and R\"uschendorf \cite{RR}, Villani \cite{Vil}
and the references therein  
for more information.
 
 \subsection{Set-up}
 \label{ss.setup}
For two closed
 sets $X, Y \subset \R^d$  set $\Om:= X \times Y$, 
$$
 \cC_b= C_b(\Om),
 \quad
 \cC_x:= C_b(X), \quad
 \cC_y := C_b(Y).
$$
The Banach lattice $\cX= \cC_b$ has the order unit
${\bf{e}}\equiv 1$.  
We fix
 $$
 \mu \in \cM_1(X), 
 \quad
 {\mbox{and}}
 \quad
 \nu \in \cM_1(Y),
$$
where $\cM_1(Z)$ is the set of all probability measures
 on a given Borel subset $Z$ of a Euclidean space.
Set
\begin{align*}
\cH_{ot}&:= \left\{ h\oplus g \ :\
h \in \cC_x, \ g \in \cC_y\quad
{\mbox{and}}\quad \mu(h)=\nu(g)=0 \right\},\\
\cQ_{ot}&:=\left\{ \eta \in (\cC^\prime_b)_+ \cap B^\prime_1
\ :\ \ \eta(f)=0,\ \
\forall f \in \cH_{ot} \right\}.
\end{align*}
By its definition,
$\cQ_{ot}$ and $\cH_{ot}$ satisfy Assumption \ref{a.main}.
$\cQ_{ot}$ also has the following well-known representation \cite{Kel}.
We provide its simple proof for completeness.

We first note that by its definition $\cQ_{ot}$ is a subset of $\cC^\prime_b(\Om)$
and any element $\varphi \in \cC^\prime_b(\Om)$
is a {\em{regular}} bounded finitely additive measure on $\Omega$. 
Moreover, $\varphi$ is countably additive if and only if
it is tight, i.e., for every $\epsilon>0$,
there is a compact $K_\epsilon \subset \Om$
such that  $|\varphi(\Om \setminus K_\epsilon)| \le \epsilon$
(see \cite{AB}).
Since the marginals $\mu$ and $\nu$ 
are countably additive measures,
we  show in the Lemma below
that the elements of $\cQ_{ot}$ are 
tight and consequently are
countably additive probability measures.

\begin{lem}
\label{l.aot}
Any $\eta \in \cC_b^\p$ belongs to
$\cA_{ot}=\cH_{ot}^\perp$ if and only if
$\eta_x= \eta(\1) \mu$ and
$\eta_y= \eta( \1) \nu$.
Moreover,
$\cQ_{ot}$ is a non-empty subset of $\cM_1(\Om)$ and
$\eta \in\cQ_{ot}$ if and only if
 $\eta_x=\mu$ and 
 $\eta_y=\nu$.  
\end{lem}

\begin{proof} 
Clearly $\mu \times \nu \in \cQ_{ot}$ and hence $\cQ_{ot}$ 
and therefore $\cH_{ot}^\perp$ are non-empty.

Let $\eta \in \cH_{ot}^\perp$.  
Then, for any
$h \in \cC_x$,
$h\oplus \0 -\mu(h) \in \cH_{ot}$.
Therefore, 
$$
0= \eta(h\oplus \0-\mu(h)) = \eta_x(h) - \eta(\1) \mu(h).
$$
Hence, $\eta_x=\eta(\1) \mu$.  Similarly, $\eta_y=\eta(\1)\nu$.
The opposite implication is immediate. 

Suppose $\eta \in \cQ_{ot}$.  Then,
$\eta(\1)=1$ and consequently, 
$\eta_x=\mu$ and $\eta_x=\nu$. 
It remains to show that $\eta$ is in $ca_r(\Om)$
or equivalently that it is countably additive.

For each $\eps >0$ choose 
compact sets $\hat K^\eps_x \subset X$,
$\hat K^\eps_y \subset Y$ so that
$$
\mu(\hat K_x^\eps), \nu(\hat K_y^\eps) > 1-\eps/2.
$$
Then, there exist $h \in \cC_x, g \in \cC_y$ 
and compact sets
$K_x^\epsilon \subset X$, $K_y^\eps \subset Y$
such that $0\le h, g \le 1$, 
$h(x)=1$ whenever $x \in \hat K_x^\eps$, $h(x)=0$
for all $x \not \in K_x^\eps$, and
$g(y)=1$ whenever $y \in \hat K_y^\eps$, $h(y)=0$
for all $y \not \in K_y^\eps$.
Set $\Omega_\eps:=K_x^\eps \times K_y^\eps$.
Then, for any  $\eta \in \cQ_{ot}$,
$$
\eta(\Omega_\eps) \ge  \eta(h \oplus g)/2.
$$
Since $g, h \le 1$, $h(x), g(y) \ge h(x)g(y)$.  Hence,
\begin{align*}
h(x)+g(y) &\ge 2 h(x) g(y)
 = h(x)+g(y) - h(x)(1-g(y))-h(y)(1-g(x)) \\
& \ge h(x)+g(y) -(1-g(y))-(1-h(x)) 
 =  2 [h(x)+g(y) -1].
\end{align*}
This implies that
$(h \oplus g)/2 \ge h \oplus \0 +\0 \oplus g-1$.
Combining all the above inequalities, 
we conclude the following for any $\eta \in \cQ_{ot}$,
\begin{align*}
\eta(\Omega_\eps) &\ge \eta(h \oplus g)/2
\ge  \eta(h \oplus \0) +\eta(\0 \oplus g) -1 =  \mu(h) + \nu(g)-1
 \\ &\ge  \mu(\hat K_\eps) + \nu(\hat K_\eps) -1 \ge  1- \eps.
\end{align*}
Hence, any $\eta \in \cQ_{ot}$ is tight.
In addition $\eta \in \cC_b^\prime$
and therefore,
it is regular and finitely additive. 
These imply that any $\eta \in \cQ_{ot}$ is a countably
additive.
\end{proof}
\vspace{2pt}

\subsection{Dual Elements}
\label{ss.elements}
Since $\cQ_{ot}$ is non-empty, by Corollary \ref{c.cont},
the duality 
holds for continuous functions, i.e.,
\begin{align*}
\dP(f;\cQ_{ot})&:= \sup_{\eta \in \cQ_{ot}} \eta(f)\\
&= \dD(f;\cH_{ot}+(\cC_b(\Om))_-)\\
&:= \inf \left\{ c \in \R\ :\
\exists h \in \cH_{ot}\ \
{\mbox{such that}}\ \ 
{\bf{c}}+h \ge f\ 
\right\},
\quad f \in \cC_b.
\end{align*} 

We continue by studying the duality in the bidual
and in $\cB_b(\Om)$.
By Theorem \ref{t.main}, the duality on
 $\cC^{\prime\prime}_b$ holds with 
 $$
 \fK_{ot}:= \left\{ \fz \in \cC_b^{\prime\prime}\ :\
 \fz(\eta) \le 0, \quad \forall \ \eta \in \cQ_{ot}\ 
 \right\}.
 $$
By Lemma \ref{l.connect}, $\fK_{ot}$ is the weak$^*$ closure
 of $\fH_{ot}+ (\cC^{\prime\prime}_b)_-$, where
 $$
\fH_{ot}:= \left\{ \fh \in \cC_b^{\prime\prime}:\ \
\fh(\eta)=0, \ \ \forall \ \eta \in \cA_{ot}=\cH_{ot}^\perp\  \right\}.
$$
  We continue by obtaining a characterization
 of  $\fH_{ot}$ and $\fK_{ot}$. 
 We then use these results to prove the duality in $\cB_b(\Om)$. 

 Towards this goal, first observe that the projection maps
 $$
  \Pi_x : \eta \in \cC_b^{\prime} \mapsto \eta_x \in \cC_x^\prime, \quad
 {\mbox{and}}
 \quad
  \Pi_y : \eta \in \cC_b^{\prime} \mapsto \eta_y \in \cC_y^\prime
  $$
 are  bounded linear maps with
operator norm equal to one.
 Also, for $\fb \in \cC_x^{\prime\prime}$, $\fc \in \cC_y^{\prime\prime}$,
 define  $ \fb \oplus \fc$ in $\cC_b^{\prime\prime}$ by
 $$
 \left(\fb \oplus \fc \right)(\eta):= 
 \fb(\eta_x)+ \fc(\eta_y), \quad
 \forall\ \eta \in \cC_b^{\prime}.
 $$

We start by proving that certain relevant sets
are regularly convex.
Recall that $\mu \in \cM_1(X)$ and $\nu \in \cM_1(Y)$
are given probability measures.  Set
\begin{align*}
\fB&=\fB_\mu:= \left\{ \fa \in \cC_b^{\prime\prime}\ :\ \exists \fb \in \cC_x^{\prime\prime}, \ \ 
{\mbox{such that}}\ \
\fa=\fb \oplus {\bf{0}}\ \ {\mbox{and}}\ \ 
\fb(\mu)=0\ \right\},\\
\fC&=\fC_\nu:= \left\{ \fa \in \cC_b^{\prime\prime}\ :\ \exists \fc \in \cC_y^{\prime\prime}, \ \ 
{\mbox{such that}}\ \
\fa={\bf{0}} \oplus \fc\ \ {\mbox{and}}\ \ 
\fc(\nu)=0\ \right\}.
\end{align*}
 
\begin{lem}
\label{l.erc}
$\fB$ and $\fC$ are regularly convex.
\end{lem}
 
\begin{proof}
Since, both $\fB$ and $\fC$ are clearly
 convex, we need to prove that they are also
 weak$^*$ closed. Since the proofs for  $\fB$
 and $\fC$ are same, we prove only $\fB$.
  $$
 \Pi^\prime_x : \cC_x^{\prime\prime} \mapsto  \cC_b^{\prime\prime}
  $$
 be the adjoint operator of $\Pi_x$.  Then, 
 \begin{align*}
 \Pi_x^\prime(\fb)(\eta)&= \fb(\Pi_x(\eta))= \fb(\eta_x)=\fb \oplus {\bf{0}} (\eta), \quad 
 \forall \fb \in \cC_x^{\prime\prime},\ \eta \in \cC_b^{\prime}.
 \end{align*}
 In particular,
 $$
 \fB= \Pi^\prime_x\left(\{\fb\in\cC_x^{\prime\prime}:\fb(\mu)=0\}\right). $$
 Additionally the set $\{\fb \in\cC_x^{\prime\prime}:\fb(\mu)=0\}$ is weak* closed. 
 Therefore, by the closed range theorem
 \cite{Rudin}[Theorem 4.14],
 the weak$^*$ closedness of
 $\fB$ is implied by the surjectivity 
 of the maps $\Pi_x$ 
 (and hence
the  closedness of its range).
 
 Indeed, fix $\beta^\prime \in 
 \cC_y^\p\cap B^\prime_+$. 
 Then, for any $\alpha \in \cC_x^\prime$,
 $$
  \Pi_x(\alpha \times \beta^\prime)= \alpha.
    $$
  Hence, $\Pi_x$ is surjective
  and hence, $\fB$ is regularly convex.
 \end{proof}
 \vspace{4pt}
 
 We continue by characterizing $\fK_{ot}$.  The following
 estimate is needed towards this result. 
 For  $R>0$, set
 $\fB_R:= \fB\cap B^{\prime\prime}_R, \quad
 \fC_R:= \fC\cap B^{\prime\prime}_R$.
 \begin{lem}
 \label{l.technical}  
 For any $R>0$, we have,
 \begin{align*}
 \left( \fB +\fC \right) \cap B^{\prime\prime}_R 
 &\subset \fB_{R} + \fC_{R},\\
 \left( \fB +\fC + (\cC^{\prime\prime}_b)_-\right) 
 \cap B^{\prime\prime}_R 
 &\subset \fB_{3R} + \fC_{3R} + 
 \left[(\cC^{\prime\prime}_b)_- \cap B^{\prime\prime}_{7R}\right].
\end{align*}
 \end{lem}
 \begin{proof} The proof
 of the first statement is simpler, so
 we only prove the second estimate.
 Also by scaling it suffices to consider the 
 case $R=1$.
 
  Set  $\fK:=  \fB +\fC + (\cC_b^{\prime\prime})_-$
  and fix $\fz_0 \in \fK \cap B^{\prime\prime}_1$. Then, there are $\fb_0 \in \cC_x^{\prime\prime}$,
 $\fc_0 \in \cC_y^{\prime\prime}$ and $\fn_0 \in 
 (\cC_b^\pp)_-$ so that
 $$
 \fz_0= \fb_0 \oplus \fc_0 + \fn_0\mbox{ and }\fb_0(\mu)=\fc_0(\nu)=0.
 $$
 Then, for any $\alpha \in \cM_1(X)$ and  $\beta \in \cM_1(Y)$,
 \begin{align*}
 \fb_0(\alpha)&= \fb_0(\alpha)+\fc_0(\nu) \ge \fz_0(\alpha \times \nu) \ge - 1\\
 \fc_0(\beta)&= \fb_0(\mu)+\fc_0(\beta) \ge \fz_0(\mu \times \beta) \ge - 1.
 \end{align*}
 Hence, $\fb_0, \fc_0 \ge -1$.
 Set,
 \begin{align*}
 \fb_1&:= \fb_0 \wedge  {\bf{2}}, 
 \quad \fb_2:= \fb_1 -{\bf{\fb_1(\mu)}} {\bf{1}},\\
 \fc_1&:= \fc_0 \wedge  {\bf{2}}, 
 \quad \fc_2:= \fc_1 -{\bf{\fc_1(\nu)}} {\bf{1}}.
 \end{align*}
 We know that $\fz_0 \le \fb_0 \oplus \fc_0$, $\fz_0 \le {\bf{1}}$
 and $\fb_0, \fc_0 \ge -{\bf{1}}$.  
  Using all these we conclude that
  $\fz_0 \le \fb_1 \oplus \fc_1$.
  Also, $\fb_1(\mu) \le \fb_0(\mu) =0$ and
  $\fc_1(\nu) \le \fc_0(\nu) =0$.  These imply that
  $\fb_1(\mu), \fc_1(\nu) \in [-1,0]$.  
  Therefore,
  $$
  \fz_0 \le \fb_1 \oplus \fc_1 \le \fb_2 \oplus \fc_2 =
   \fb_2 \oplus {\bf{0}} + {\bf{0}} \oplus \fc_2.
   $$
   Moreover, 
   $$
   \fb_2 = \fb_0 \wedge  {\bf{2}} - \fb_1(\mu){\bf{1}} 
   \le {\bf{2}} + {\bf{1}} \le {\bf{3}}, \quad
   \fb_2 = \fb_0 \wedge  {\bf{2}} - \fb_1(\mu) {\bf{1}} 
   \ge \fb_0 \wedge {\bf{2}} \ge -{\bf{1}}.
   $$
   Hence, $ \fb_2 \oplus {\bf{0}}  \in \fB_3$.  Similarly,
   ${\bf{0}} \oplus \fc_2 \in \fC_3$.
    
 Finally, set $\fn_2 :=\fz_0 -\fb_2  \oplus \fc_2$.  It is now
 clear that $ 0\ge \fn_2 \in B^{\prime\prime}_7$.
  \end{proof}
 \vspace{4pt}
  \subsection{Duality in the bidual}
  \label{ss.bidual}
 We now obtain a complete characterization
 of the dual elements.  
  \begin{prp}
 \label{p.kot}
   We have $\fH_{ot}= \fB + \fC$ and
 $\fK_{ot} = \fB +\fC + (\cC_b^{\prime\prime})_-=
\fH_{ot}+ (\cC^{\prime\prime}_b)_-$. In particular,
for $\fa \in \cC_b^{\prime\prime}$,
\begin{align*}
\dD_{ot}(\fa)&:= \dD(\fa;\fK_{ot}) = \dD(\fa;\fB +\fC + (\cC_b^{\prime\prime})_-) \\
&= \min \left\{ \ c \in  \R\ : \
\exists\ \fb \in \fB, \  \fc \in \fC,\ 
{\mbox{such that}}\ \
c \ {\bf{1}}+ \fb \oplus \fc  \ge \fa\ \right\}\\
&=\dP_{ot}(\fa):=\dP(\fa;\cQ_{ot})
\end{align*}
 \end{prp}
\begin{proof} Set
 $\fK:=  \fB +\fC + (\cC_b^{\prime\prime})_-$.
 It is clear that $\fK \subset \fK_{ot}$. 
  \vspace{4pt}

{\em{Step 1.}} ({\em{Regular convexity  of $\fK$}}). 

 In view of Lemma \ref{l.technical},
 for any $R>0$,
 $$
 \fK \cap B^{\prime\prime}_R \subset \fK_R:= 
 \fB_{3R} + \fC_{3R} + (\cC_b^{\prime\prime})_- \cap B^{\prime\prime}_{7R}.
 $$
It is shown in Lemma \ref{l.erc}
that $\fB$ and $\fC$ are regularly convex.
It is clear that $(C_b^\pp)_-$ is also
regularly convex.
Hence, $\fB_{3R}$ and $\fC_{3R}$ and as well as
$\cC^{\prime\prime}_-\cap B^{\prime\prime}_{7R}$ are regularly
convex for each $R>0$.  By Theorem 7 of \cite{KS},
$\fK_R$ is then regularly convex.  Consequently, by Lemma \ref{l.condition},
$\fK$ is regularly convex. 
  \vspace{4pt}

{\em{Step 2.}} ($\fH_{ot}= \fB + \fC$). 

Let $\fa_0 \in \fH_{ot} = \cA_{ot}^\perp$.  For 
$\alpha \in \cC_x^\prime$, $\beta \in \cC_y^\prime$ define
$$
\fb_0(\alpha):= \fa_0(\alpha \times \nu), \quad
\fc_0(\beta):= \fa_0(\mu \times \beta).
$$
Then, for any $\eta \in \cC_b^\prime$,
$$
 (\fb_0 \oplus \fc_0)(\eta)-\fa_0(\eta)= \fa_0(\tilde \eta),
\quad
{\mbox{where}}
 \quad
 \tilde \eta= \eta_x \times \nu+\mu \times \eta_y-\eta.
 $$
 One can directly check that $\tilde \eta_x=\mu$
and $\tilde \eta_y=\nu$.  Hence, by Lemma \ref{l.aot},
$\tilde \eta \in \cA_{ot}$
and consequently $\fa_0(\tilde \eta)=0$.  This shows that
$\fa_0=\fb_0 \oplus \fc_0$ or equivalently
$\fH_{ot} \subset \fB + \fC$. The opposite inclusion
is immediate.
\vspace{4pt}

{\em{Step 3.}} ({\em{Conclusion}}). 

In view of Lemma \ref{l.connect}, $\fK_{ot}=\fK_{\cQ_{ot}}$ is the
weak$^*$ closure of $\fH_{ot} +(\cC^{\prime\prime}_b)_-$.  Also by the second step,
$\fH_{ot} +(\cC^{\prime\prime}_b)_- 
=\fB +\fC + (\cC_b^{\prime\prime})_- =:\fK$.  In the first step,
 it is shown that $\fK$ is regularly convex and consequently is
weak$^*$ closed.  Therefore, the weak$^*$ closure of 
$\fK=\fH_{ot} +(\cC_b^{\prime\prime})_-$ is equal to itself.

The duality statement follows from  Theorem \ref{t.main}. 
\end{proof}
\vspace{2pt} 
\subsection{Duality in $\cB_b(\Om)$}
\label{ss.lsonsuz}
We close this section by proving the duality
on $\cB_b(\Om)$, the set of all bounded,
Borel measurable functions.
Let $\cN_{ot}$ be the set of all $\cQ_{ot}$ polar sets, i.e.,
a Borel set $N$ is $\cQ_{ot}$ polar when $\eta(N)=0$ for 
every $\eta \in \cQ_{ot}$.  Further
let $\cZ_{ot}$ be the set of all functions $\zeta \in \cB_b(\Om)$
such that $\{\zeta \neq 0\} \in \cN_{ot}$. Finally, set
\begin{align*}
\cH_{ot}^\infty&:=\left\{ h \oplus g\ : \ h\in \cB_b(X),\ 
g \in \cB_b(Y)\quad
{\mbox{and}}\quad \mu(h)=\nu(g)=0 \right\} \\
\hat \cH_{ot}^\infty &:=\left\{ \zeta + h \oplus g \ : \ \zeta \in \cZ_{ot},\
h \oplus g \in \cH_{ot}^\infty\ \right\}.
\end{align*}
We follow the lecture notes of Kaplan \cite{Kap}
to characterize the dual set.

\begin{prp}[{\bf{Duality in $\cB_b(\Om)$}}]
\label{p.lsonsuz}
For any $\xi \in \cB_b(\Om)$,
\begin{align*}
\dP_{ot}(\xi)&= \dD\left(\xi; \hat \cH_{ot}^\infty+ \cB_b(\Om)_-\right)\\
&= \min \left\{ \ c \in  \R\ : \
\exists\ \zeta \in \cZ_{ot},\
h\oplus g \in \cH_{ot}^\infty,\ 
{\mbox{such that}}\ \
c \ {\bf{1}}+\zeta+ h \oplus g  \ge \xi\ \right\}. 
\end{align*}
\end{prp}
\begin{proof} In view of Theorem \ref{t.main},
it suffices to show the following,
$$
\cK_{ot}:= \left\{ \xi \in \cB_b(\Om) :
\dP_{ot}(\xi) \le 0\right\}
= \hat \cH_{ot}^\infty+ \cB_b(\Om)_-.
$$
It is clear that $\hat \cH_{ot}^\infty + \cB_b(\Om)_-\subset \cK_{ot}$.  
We continue by proving the opposite inclusion.
\vspace{2pt}

\noindent
{\em{Step 1.}} 
Fix $\xi \in \cK_{ot}$.  
Then, $\fI(\xi) \in \fK_{ot}$
 and in view of Proposition \ref{p.kot},
there are $\fb \in \cC_x^\pp$, $\fc \in \cC_x^\pp$ satisfying,
$\fI(\xi) \le \fb \oplus \fc$ and 
$\fb(\mu)=\fc(\nu)=0$.
Consider the map,
$$
\cG_\fb: H\in \cL^1(\Om,\mu) \mapsto \cG_\fb(H):= \fb(\mu_H),
\quad
{\mbox{where}}\quad 
\mu_H(A):= \int_A H(x) \ d\mu(x).
$$
It is immediate that $\cG_\fb$ is a bounded linear map 
on $\cL^1(\Om,\mu)$.  Hence, there exists $h_\xi \in \cL^\infty(X,\mu)$
satisfying,
$$
\cG_\fb(H)= \int_X h_\xi(x) H(x) d\mu(x) = \mu_H(h_\xi).
$$
We rewrite the above identity as follows,
$$
\fb(\alpha)= \alpha(h_\xi), \quad \forall \alpha \in \cC_x^\p\ \
{\mbox{and}}\ \  |\alpha| \ll \mu.
$$
Similarly there is
$g_\xi \in \cL^\infty(Y,\nu)$
such that
$$
\fc(\beta)= \beta(g_\xi), \quad \forall\beta \in \cC_y^\p \ \
{\mbox{and}}\ \  |\beta| \ll \nu.
$$
We fix pointwise representatives of 
$h_\xi$ and $g_\xi$ and set
$$
\cN_\xi:=\left\{ \om \in \Om \ :\
\zeta_\xi(\om) >0 \ \right\}, \quad
{\mbox{where}}
\quad
\zeta_\xi(\om):= \left[ \xi- h_\xi\oplus g_\xi\right](\om).
$$
\vspace{4pt}

\noindent
{\em{Step 2.}}
In this step, we show that $(\zeta_\xi)^+ \in \cZ_{ot}$.

Fix $\eta \in \cQ_{ot}$.  Let 
$\{\nu(x,\cdot)\}_{x\in X} \subset \cM_1(Y)$
and $\{\mu(y,\cdot)\}_{y\in Y} \subset \cM_1(X)$
be Borel measurable families of
probability measures 
satisfying
$$
\int_\Om  \zeta d\eta
 = \int_X \left[\int_Y \zeta(x,y) \nu(x, dy)\right] \mu(dx)
 = \int_Y \left[\int_X \zeta(x,y) \mu(y, dx)\right] \nu(dy),
$$
for every $\zeta \in \cB_b(\Om)$. 
Since $\eta$ is a probability measure,
we may define $\eta_\xi \in \cC_b^\p$ by,
$$
\eta_\xi(f):= \int_{\cN_\xi} \ f \ d\eta,
\quad \forall f \in \cC_b.
$$
Then, one can directly show that for every Borel
set $A \subset X$,
$$
\left(\eta_\xi\right)_x(A)= \int_A 
\left[\int_Y \chi_{\cN_\xi}(x,y) \nu(x, dy)\right] \mu(dx).
$$
Since $\left(\eta_\xi\right)_x$ is absolutely
continuous with respect to $\mu$, the construction of $h_\xi$
implies that
$$
(\fb\oplus {\bf{0}})(\eta_\xi) = \fb(\left(\eta_\xi\right)_x)
= \left(\eta_\xi\right)_x(h_\xi)
=\int_\Om h_\xi(x) \chi_{\cN_\xi}(x,y) \eta(dx,dy).
$$
Similarly, one can show that,
$$
({\bf{0}}\oplus \fc)(\eta_\xi) = \fb(\left(\eta_\xi\right)_y)
= \left(\eta_\xi\right)_y(g_\xi)
= \int_\Om g_\xi(y) \chi_{\cN_\xi}(x,y) \eta(dx,dy).
$$
These imply that
$$
0 \le (\fb \oplus \fc - \fI(\xi))(\eta_\xi)
= \int_\Om [(h_\xi \oplus g_\xi) -\xi]\ \chi_{\cN_\xi} d\eta
=\int -\ \zeta_\xi \chi_{\{ \zeta_\xi >0\}} d\eta
= - \eta((\zeta_\xi)^+) \le 0.
$$
Therefore, $\eta((\zeta_\xi)^+)=0$ for every $\eta \in \cQ_{ot}$
and $(\zeta_\xi)^+\in \cZ_{ot}$.
\vspace{2pt}

\noindent
{\em{Step 3.}} 
The definition of $\zeta_\xi$ implies that
$$
\xi= [\xi-h_\xi\oplus g_\xi] + h_\xi\oplus g_\xi
\le (\zeta_\xi)^+ + h_\xi\oplus g_\xi.
$$
Since $\mu(h_\xi)=\nu(g_\xi)=0$, this
proves that $\xi \in \hat \cH_{ot}^\infty$.
\end{proof}
\vspace{4pt}
\begin{rem}
\label{r.disclaimer}
{\rm{ The above dual problem with the hedging set
$\hat \cH_{ot}^\infty$ can also be seen
as  quasi-sure super-replication in the sense
defined in \cite{GLT,STZ,STZ1}.  Indeed,
we say two functions
$\ell, \xi$ satisfy $\ell \ge \xi$, $\cQ_{ot}$ quasi-surely 
and write $\ell\ge \xi,$ $\cQ_{ot}-q.s.$,
if $\eta(\{  \ell < \xi \}) =0$ for every $\eta \in \cQ_{ot}$
or equivalently, if the set $\{\ell < \xi \}$ is a $\cQ_{ot}$ polar set.
Then, we have the following immediate representation
of the dual problem,
$$
\dD(\xi;\hat \cH^\infty_{ot}+\cB_b(\Om)_-)
=\inf\{ c \in \R\ :\
\exists h\oplus g \in \cH_{ot}^\infty, \ \ 
{\mbox{s.t.}}\ \ 
c \1 +h\oplus g \ge \xi, \ \cQ_{ot}-q.s.\ \}.
$$

In the classical paper of Kellerer \cite{Kel},
the duality is proved with the hedging set $\cH_{ot}^\infty$
without augmenting it with the $\cQ_{ot}$
polar sets.  In particular, Kellerer's duality result
shows that every $\cQ_{ot}$
polar set $N$ is dominated from above 
by the sum of a $\mu$ null set $A$ and a 
$\nu$ null set $B$, i.e.,
\begin{equation}
\label{e.qnulls}
\chi_N(x,y) \le \chi_A(x)+ \chi_B(y), \quad
{\mbox{and}}
\quad
\mu(A)=\nu(B)=0.
\end{equation}
Kellerer proved the above result
by using the classical Choquet 
capacity theory.  Indeed,
the result would follow
if the primal functional
is shown to have certain
regularity properties
as assumed in the 
Choquet Theorem. 
Then one shows these properties
using the Lusin theorem and other approximations.

On the other hand, our results imply that 
there are $\fb \in \cC_x^\pp,
\fc \in \cC_y^\pp$ satisfying,
$$
\fI(\chi_N) \le \fb \oplus \fc, \quad
{\mbox{and}}
\quad
\fb(\mu)=\fc(\nu)=0.
$$
To prove \reff{e.qnulls} from the above statement
is an interesting analytical question.
In particular,
it is not clear if the Kellerer approach 
via the Choquet capacity theory
is the only possible way.
We leave these questions to further research
and do not pursue them here.
}}
\qed
\end{rem}
\vspace{4pt}

\section{Constrained Optimal Transport}
\label{s.cot}

In this section, we investigate an extension of
the classical optimal transport.  
In this extension,
we are given a finite subset 
$\{f_1,\ldots, f_N\} \subset \cC_b$.  
We set $\cQ_0:= \cQ_{ot}$ and for $k=1,\ldots, N$,
define,
$$
\cQ_k:= \cH_k^\perp \cap \partial B^\p_+,
\quad
{\mbox{where}}
\quad
\cH_k:=  \left\{ f + \sum_{i=1}^k a_i f_i\ :\ f \in \cH_{ot}, \ \ a_i \in \R\ 
\right\}.
$$
We make the following structural assumption.
\begin{asm}
\label{a.structure}
For $k=1,\ldots,N$, we assume that
$$
\inf_{\eta \in \cQ_{k-1}} \eta(f_k) < 0 <
\sup_{\eta \in \cQ_{k-1}} \eta(f_k).
$$
\end{asm}
\vspace{2pt}

\begin{rem}
\label{r.neden}
{\rm{The above assumption is equivalent
to the following,
$$
-\dP(-f_k;\cQ_{k-1}) <0 <\dP(f_k;\cQ_{k-1}) ,
$$
for every $k=1,2,\ldots,N$.  In this assumption, 
the value
zero is not important.  Indeed, if $\bar {f_k}$
satisfies
$$
-\dP(-\bar{f_k};\cQ_{k-1}) <\dP(\bar{f_k};\cQ_{k-1}),
$$
then there exists a constant $b_k$ so that 
$f_k:= \bar{f_k}-b_k$ satisfies the Assumption \ref{a.structure}.

Moreover, Assumption \ref{a.structure} implies that $\cQ_k$
is non-empty and satisfies the Assumption \ref{a.main}.
Conversely, the inequality
$$
-\dP(-f_k;\cQ_{k-1}) \le 0  \le \dP(f_k;\cQ_{k-1}),
$$
is necessary for $\cQ_k$ to be non-empty.}}
\qed
\end{rem}
\vspace{2pt}

Set $\fH_0:=\fH_{ot}$,
$\cQ_{cot}:= \cQ_N$, 
$\cH_{cot}:= \cH_N$,
$\fH_{cot}:= (\cH_{cot}^\perp)^\perp$,
$$
\fK_{cot}:= \left\{ \fz \in \cC_b^\pp\ :\ \fz(\eta)\le 0, 
\  \forall \eta \in \cQ_{cot}\ \right\}, 
$$
and let $\fF_k$ be the subspace spanned by $\fI(f_k)$.

\begin{prp}
\label{p.cot}
Suppose that Assumption 
\ref{a.structure} holds.
Then, 
$$
\fH_{cot}= \fH_{ot}+ \sum_{i=1}^N \fF_i, \qquad
\fK_{cot}=\fH_{cot}+\left(\cC_b^\pp\right)_-.
$$
In particular,
\begin{align*}
\dP(\fa;\cQ_{cot})&= \dD(\fa;\fH_{cot}+\left(\cC_b^\pp\right)_-)\\
&= \min \left\{ c \in \R\ :\
\exists \fh \in \cC_x^\pp,\ \ \fg \in \cC_y^\pp, \ \ a_1,\ldots, a_n \in \R \right. \\ 
&\hspace{100pt}\left.{\mbox{such that}}\ \ 
\fh\oplus \fg +\sum_{k=1}^N a_k \fI(f_k) \ge \fa\ \right\}.
\end{align*} 
\end{prp}

\begin{proof} Set 
$$
\fH_k:= (\cH_k^\perp)^\perp,\quad
\fK_k:=\{ \fz\in \cC_b^\pp\ :\ \fz(\eta) \le 0, \ \forall \eta \in \cQ_k\ \},
\qquad k=1,\ldots,N.
$$
It is clear that $\fH_{cot}= \fH_N$ and $\fK_{cot}=\fK_N$.
  
We shall prove  by induction that
$$
\fH_k= \fH_{ot}+ \sum_{i=1}^k
\fF_i, \qquad
\fK_k=\fH_k+\left(\cC_b^\pp\right)_-, \qquad
k=1,\ldots, N.
 $$
We know that 
the above statements
hold for $k=0$.  Indeed
$\fH_0=\fH_{ot}$ by definition
and by Proposition \ref{p.kot},
$\fK_0=\fH_{ot}+\left(\cC_b^\pp\right)_- $.
Suppose now that the claim holds for $k-1$
with some $k\ge 1$.
Set $\fH:= \fH_{k-1}+\fF_k$ and
$\fK:= \fK_{k-1}+\fF_k$. 
We claim that both $\fH$ and $\fK$ are
weak$^*$ closed.
Since both proofs are similar, we prove
only the second statement
 by an application of Lemma \ref{l.condition}.

Fix an arbitrary $\fa \in \fK \cap B^\pp_1$.  Then, there are
$\fz_{k-1} \in \fK_{k-1}$ and $a_k \in \R$ satisfying,
$$
\fa= \fz_{k-1} + a_k  \fI(f_k).
$$
Since $\fK_{k-1}=\fH_{k-1}+\left(\cC_b^\pp\right)_-$,
there are $\fh_{k-1} \in \fH_{k-1}$ and $\fn_{k-1} \le 0$ such that
$\fz_{k-1}=\fh_{k-1}+\fn_{k-1}$.
Also, Assumption \ref{a.structure} states that
$$
\underline{p}:= \inf_{\eta\in \cQ_{k-1}}\eta(f_k) 
<\ 0 \ <
\overline{p}:= \sup_{\eta\in \cQ_{k-1}}\eta(f_k).
$$

We analyse two cases separately.
First, suppose that $a_k \ge 0$. Then,
$$
a_k \ \underline{p}  \ge \inf_{\eta \in \cQ_{k-1}} 
\left(\fz_{k-1}+a_k \fI(f_k)\right)(\eta) =  \inf_{\eta \in \cQ_{k-1}} \fa(\eta)
\ge -1.
$$
Since $\underline{p} <0$,
$a_k \le 1/(-\underline{p})$.

Next, suppose that $a_k  \le 0$. Then,
$$
a_k \ \overline{p}  \ge \sup_{\eta \in \cQ_{k-1}} 
\left(\fz_{k-1}+a_k \fI(f_k)\right)(\eta) =  \sup_{\eta \in \cQ_{k-1}}  \fa(\eta)
\ge -1.
$$
Hence
$a_k \ge -1/(\overline{p})$.

Combining both cases, we conclude that
$$
|a_k| \le c^*_k:=\max\left\{ \frac{1}{\overline{p}}\ ,\ 
\frac{1}{-\underline{p}}\ \right\}.
$$
Therefore, 
$$
\left\| \fz_{k-1}\right\|_\infty \le 1+ c^*_k \|f_k\|_\infty.
$$
We now apply Lemma \ref{l.condition} of the appendix
to conclude that $\fK$ is regularly convex. Hence,
$\fK= \fK_k$.  

Since $\fK$ is defined to be 
$\fK_{k-1}+ \fF_k$,
by the induction hypothesis,
$$
\fK_{k-1}= \fH_{k-1} + \left(\cC_b^\pp\right)_-
\quad
\Rightarrow
\quad
\fK_k=\fK= \fH_{k-1} +  \fF_k + \left(\cC_b^\pp\right)_-.
$$
A similar induction argument shows that 
$$
\fH_k= \fH_{k-1} +  \fF_k.
$$
Hence, we conclude that
$$
\fK_k= \fH_k + \left(\cC_b^\pp\right)_-\quad
{\mbox{and}}\quad
\fH_k= \fH_{ot}+\sum_{i=1}^k \fF_i.
$$
The duality statement follows from the above
characterizations, Theorem \ref{t.main} and the fact that $\cQ_{cot}$ satisfies the Assumption \ref{a.main}.
\end{proof}
\vspace{2pt}

We can prove the duality in $\cB_b(\Om)$
as in Proposition \ref{p.lsonsuz}. We state
this result without a proof for completeness.
Let $\cZ_{cot}$ be the set of all bounded functions $\zeta$
such that $\eta(\{ \zeta \neq 0\})=0$ for every
$\eta \in \cQ_{cot}$.  Set
$$
\hat \cH_{cot}^\infty:=
\left\{ \zeta + h\oplus g + \sum_{i=1}^N a_i f_i\ :\ 
\zeta \in \cZ_{cot},\
h\oplus g \in \cH_{ot}^\infty, \ \ a_i \in \R\ 
\right\}.
$$

\begin{cor}
\label{c.lsonsuz}
For any $\xi \in \cB_b(\Om)$,
$$
\dP_{cot}(\xi)= \dD\left(\xi; 
\hat \cH_{cot}^\infty+ \cB_b(\Om)_-\right).
$$
\end{cor}
\vspace{2pt}

\section{Martingale Measures and super-martingales}
\label{s.mart}

Suppose that $X=Y \subset \R^d$ be convex and closed sets.
Recall that
$$
\cC_\ell= C_{\ell}(\Om):= \left\{ f \in C(\Om)\ :\
\| f\|_{\ell} <\infty \ \right\},
$$
with the weighted norm
defined in \reff{e.lnorm}.  
\vspace{2pt}

\subsection{Definitions}
\label{ss.definitions}

Define  a linear functional $T: C_b(X;\R^d) \mapsto \cC_\ell$ by
$$
T(\gamma)(x,y):= \gamma(x) \cdot (x-y), \ \ 
\forall \ (x,y)\in \Om\ \ \gamma \in C_b(X;\R^d).
$$
It is clear that $T(\gamma) \in \cC_\ell$ and
the adjoint 
$T^\prime : \cC_\ell^\prime \mapsto C_b^\prime(X;\R^d)$  satisfies,
$$
T^\prime(\eta)(\gamma)= \eta(T(\gamma)),\quad
\forall  \  \om=(x,y) \in \Om, \ \ \gamma \in C_b(X;\R^d),
\ \ \eta \in C_\ell^\p.
$$
Let $T^{\prime\prime} : 
C_b^\pp(X;\R^d) \mapsto \cC_\ell^{\prime\prime}$
to be the adjoint of $T^\prime$. 
Then,
$$
T^{\prime\prime}(\fI(\gamma)) = \fI(T(\gamma)),
\quad
\forall \ \gamma \in C_b(X;\R^d),
$$
where $\fI$ is the canonical
map of $C_b(X;\R^d)$ into $C_b^\pp(X;\R^d)$.
We then
define a subset $\fD \subset \cC_\ell^{\prime\prime}$ by,
 $$
 \fD:= \left\{ \ff \in  \cC_\ell^{\prime\prime}\ :\ 
 \exists \fg \in C_b^{\prime\prime}(X;\R^d)\ {\mbox{such that}}\ \ 
 \ff= T^{\prime\prime}(\fg) \right\}.
 $$
Equivalently, $\fD$ is the range of the adjoint operator $T^{\prime\prime}$.
Finally, set
\begin{align*}
\fM&:= \cM^\perp,
\quad
{\mbox{where}}
\quad
\cM:= \left\{ \eta \in \cC^\prime_l\ :\ \eta(T(\gamma))=0, \
\ \forall \ \gamma \in C_b(X;\R^d)\ \right\},\\
\fS&:= \left\{ \fh  \in \cC^\pp_l\ :\ \ \fh(\eta) \le 0, \
\ \forall \ \eta  \in \cM \cap (\cC_\ell^\p)_+\ \right\}.
\end{align*}
It is clear that
$\fD \subset \fM$ and $\fD + (\cC_\ell^\pp)_- \subset \fS$.
\vspace{2pt}

\begin{dfn}
\label{d.mm}
{\rm{Any element $\eta$ of $\cM$ is called a}} martingale measure,
{\rm{ any $\fm \in \fM$ a}} martingale, {\rm{and any}}
$\fh \in \fS$ a super-martingale.
\end{dfn}
\vspace{2pt}

\subsection{The case $X=Y=\R^d$}
\label{ss.case}

 The following result characterizes the sets defined above
 and also motivates the terminology used in that definition.

\begin{prp}
\label{p.drc} Let $X=Y= \R^d$.
Then, 
$\fD$  and $\fD + (\cC^{\prime\prime}_l)_-$ 
are regularly convex. In particular,
$$
\fD = \fM\quad
{\mbox{and}}
\quad
\fD +  (\cC^{\prime\prime}_\ell)_-= \fS.
$$
Hence, in $\Om= \R^d \times \R^d$,
any martingale has the form $T^{\pp}(\fg)$ and $\fh \in \cC_\ell^\pp$
is a super-martingale 
if and only if it is dominated by a martingale.
\end{prp}

\begin{proof} We again use the closed range theorem to prove this result.
\vspace{2pt}

{\em{Step 1.}} ({\em{Range of}} $T${\em{is closed}} ).

Let $\gamma_n$ be a sequence in $C_b(X;\R^d)$
so that $\xi_n=T(\gamma_n)$ is strongly convergent
to $\xi \in \cC_\ell$.
For each $x \in X$, and positive integers $n, m$, set 
$$
y_{n, m}(x):= x-\frac{\gamma_n(x) -\gamma_m(x)}{|\gamma_n(x)-\gamma_m(x)|}.
$$
Then,
\begin{align*}
 \left| \gamma_n(x)- \gamma_m(x)\right| &=
 \left| \xi_n(x,y_{n, m}(x)) - \xi_m(x,y_{n, m}(x)) \right|  \\
& \le \| \xi_n - \xi_m \|_{\cC_\ell} \ [2+ |x| + |y_{n,m}(x)|] 
 \le 3 [1+ |x|] \ \| \xi_n - \xi_m \|_{\cC_\ell} \ .
\end{align*}
Therefore, $\{\gamma_n\}$ is locally Cauchy in $C_b(X,\R^d)$.
Consequently, as $n$ tends to infinity,
$\gamma_n$ converges locally
uniformly to a continuous function $\gamma \in C(X;\R^d)$. 
Moreover, it is clear that $\xi = T(\gamma)$.  
We claim that $\gamma $
is bounded.  Indeed, set
$$
y(x,\lambda):= x - \lambda \ \frac{\gamma(x)}{|\gamma(x)|},
\quad \ \lambda >0, \ x \in X, \ \gamma(x)\neq 0,
$$
and set $y(x,\lambda)=0$ when $\gamma(x)=0$.
Then, we directly estimate that,
$$
| \gamma(x)| = \frac{1}{\lambda}\ \xi(x,y(x,\lambda))
 \le   \frac{1}{\lambda}\ \|\xi\|_{\cC_\ell}\ [2+|x|+|y(x,\lambda)|]
 \le   \frac{1}{\lambda}\ \|\xi\|_{\cC_\ell}\ [2+2 |x|+\lambda].
$$
We let $\lambda $ to infinity to arrive at the following estimate,
\begin{equation}
\label{e.Test}
\| \gamma\|_{C_b(X;\R^d)} \le 
\left\| \xi \right\|_{\cC_\ell}= 
\left\| T(\gamma) \right\|_{\cC_\ell}, 
\quad \forall \ \gamma \in C_b(X;\R^d).
\end{equation}
\vspace{2pt}

{\em{Step 2.}} ($\fD= \fM$).

The previous step shows that the range of $T$ is closed. Then,
by the closed range Theorem \cite{Rudin}[Theorem 4.14], the range of $T^\prime$
is weak$^*$  and also norm closed. We now apply the same theorem
to $T^\prime$ to conclude that the range of $T^{\prime\prime}$ is weak$^*$ closed.
Since $\fD$ is defined as the range of $T^{\prime\prime}$ and since it is linear,
we conclude that it is regularly convex.  Hence, $\fM=\fD$.
\vspace{2pt}

{\em{Step 3.}} ({\em{A map}}).

Define a linear map $L : \cC_\ell \mapsto \cC_\ell$ by
$$
L(f)(x,y):= f(x,2x-y), \quad
\forall \ (x,y) \in \Om, \ \ f \in \cC_\ell.
$$
Then, one can directly verify that
for any $(x,y) \in \Om$ and $\gamma \in C_b(X;\R^d)$
 the following identity holds:
$$
L(T(\gamma))(x,y) =
T(\gamma)(x,2x-y) = \gamma(x)\cdot(x-(2x-y)) = - T(\gamma)(x,y).
$$
We use this with
$\eta \in \cC_\ell^\prime$ and $\gamma \in C_b(X;\R^d)$
to arrive at
\begin{align*}
T^{\prime\prime}(\fI(\gamma))(\eta) &=
\eta(T(\gamma))= - \eta(L(T(\gamma))) =- L^\prime(\eta)(T(\gamma))\\
 & = - T^{\prime\prime}(\fI(\gamma))(L^\prime(\eta)).
 \end{align*}
By the weak$^*$ density of $\fI(C_b(X;\R^d))$ in $C_b^{\prime\prime}(X;\R^d)$,
we conclude that the above holds for any element of $C_b^{\prime\prime}(X;\R^d)$, i.e.,
$$
T^{\prime\prime}(\fg)(\eta) = 
- T^{\prime\prime}(\fg)(L^\prime(\eta)), \quad
\forall \ \fg \in C_b^{\prime\prime}(X;\R^d), 
\ \eta \in \cC_\ell^\prime.
$$
Moreover, for any $f \in \cC_\ell$ and $(x,y) \in \Om$,
$$
\left| L(f)(x,y)\right| =\left| f(x,2x-y)\right| \le
\| f\|_{\cC_\ell} \ell(x,2x-y) \le 
3 \| f\|_{\cC_\ell} \ell(x,y).
$$
Hence,  for any $\eta \ge 0$, 
$$
\| L^\prime(\eta)\|_{\cC_\ell^\prime} \le 3 \|\eta\|_{\cC_\ell^\prime} .
$$
Also, $L^\prime(\eta) \ge 0$ whenever $\eta \ge 0$. 
\vspace{4pt}

{\em{Step 4.}} ({\em{An estimate}}).

Let  $\fa$ be an arbitrary element of $\fD+(\cC_\ell^{\prime\prime})_-$.
Then, there exists $\fg \in C_b^{\prime\prime}(X;\R^d)$ such that
$T^{\prime\prime}(\fg)(\eta) \ge \fa(\eta)$.
For any 
$\eta \in (\cC_\ell^\prime)_+$, the previous step implies the following,
$$
T^{\prime\prime}(\fg)(\eta)  =- T^{\prime\prime}(\fg)(L^\prime(\eta)) 
 \le - \fa(L^\prime(\eta)) \le 3 \|\fa\|_{\cC^{\prime\prime}} \|\eta\|_{\cC_\ell^\prime}.
 $$
We also have,
$$
T^{\prime\prime}(\fg)(\eta) \ge 
\fa(\eta) \ge - \|\fa\|_{\cC^{\prime\prime}} \|\eta\|_{\cC_\ell^\prime}.
$$
Hence,
$$
\left\| T^{\prime\prime}(\fg) \right\|_{\cC_\ell^{\prime\prime}} \le 3 \| \fa\|_{\cC_\ell^{\prime\prime}},
\quad
{\mbox{whenever}}
\quad
T^{\prime\prime}(\fg) \ge \fa.
$$

We summarize the above estimate into the following
$$
\left(\fD+(\cC_\ell^{\prime\prime})_-\right) 
\cap B^{\prime\prime}_{\ell,R}
\ \subset \ \fD \cap B^{\prime\prime}_{\ell,3R} 
+ (\cC_\ell^{\prime\prime})_- \cap B^{\prime\prime}_{\ell,4R}, \quad
\forall \ R>0,
$$
where
$$
B^{\pp}_{\ell,R} = \left\{ \ \fa \in C^{\pp}_\ell \ 
:\ \|\fa \|_{C^{\pp}_\ell} \le 1\ \right\}.
$$
\vspace{2pt}

{\em{Step 5.}} ($\fD+(\cC_\ell^{\prime\prime})_-$ {\em{is regularly convex}} ).

Since $\fD$ is proved to be regularly convex,
by the previous step and Lemma \ref{l.condition}
of the Appendix, we conclude that $\fD+(\cC_\ell^{\prime\prime})_-$
is also regularly convex.  Therefore, $\fD+(\cC_\ell^{\prime\prime})_-$
is equal to its regularly convex envelope $\fS$.
\end{proof}
We now give without proof the following corollary 
which is a direct consequence of the regular convexity of 
$\fD+(\cC_\ell^{\prime\prime})_-$.
\begin{cor}
For $\fa \in \cC_{\ell}^{\prime\prime}$,
\begin{align*}
\dP(\fa; \cM \cap  \partial B^\p_+)& = \dD(\fa;\fD+(\cC_\ell^{\prime\prime})_-) \\
&= \min \left\{ \ c \in  \R\ : \
\exists \fg \in C_b^{\prime\prime}(X;\R^d),\ 
{\mbox{such that}}\ \
c \ {\bf{1}}+ T^{\prime\prime}(\fg)  \ge \fa\ \right\}. 
\end{align*}
\end{cor}
\vspace{2pt}
\section{Martingale optimal  transport}
\label{s.mot}
 
In this example, we take 
$X= Y = \R^d$ and 
set 
$\cX= \cC_\ell$ with
the unit element,
 $$
 {\bf{e}}(x,y)= \ell(x,y):= \ell_X(x)+\ell_Y(y)=
 (1+ |x|)+(1+|y|).
 $$
 We also use the notation,
 $$
 \cC_{X}:=\cC_{\ell_X}(X), \quad
 \cC_{Y}:=\cC_{\ell_Y}(Y).
 $$
 
 \subsection{Set-up}
 \label{ss.motsu}
 
 As in Section \ref{s.ot}, we fix $\mu, \nu$.  We assume that
 they are in convex order.
 For any $\alpha \in \cC_X, \beta \in \cC_Y$, set
$$
m^*:= m^*_X+m_Y^*, \quad
m^*_X:= \mu(\ell_X), \quad
m^*_Y:=\nu(\ell_Y).
$$ 
 Recall that the set of
 martingale measures $\cM$ is defined 
 in Section \ref{s.mart}
 as the annihilators of the range of
 the map $T$ again introduced in that section.
 Set,
 $$
 \cH_{mot}:=
 \left\{ h \oplus g  + T(\gamma) \ :
 h \in  \cC_X,\ g \in \cC_Y, \ \gamma \in C_b(X;\R^d)\ 
 \right\},
 $$
 $$
 \hat \cQ_{mot}:= \cH^\perp_{mot} \cap \partial B^\p_+.
 $$
 Then, any $\eta \in \hat \cQ_{mot}$ satisfies $\eta \in ca_r^+(\Om)$
 and $\eta_x=\eta(\Om) \mu$, $\eta_y=\eta(\Om) \nu$.  In particular,
 $$
 \eta({\bf{e}}) = \eta(\ell)= \int_\Om \ell(x,y) \eta(dx,dy) = 
 m^*  \eta(\Om)=1.
 $$
Set, 
$$
\cQ_{mot}:= \left\{\eta \in \cC_\ell^\p\ :\
\eta/m^*\in \hat \cQ_{mot}\ \ \right\}.
$$
Then, $\cQ_{mot} = \cQ_{ot}\cap \cM$.
We also note that, 
by Strassen's Theorem \cite{S}, $\cQ_{mot}$ is non-empty
if and only if $\mu$ and $\nu$ 
are in convex order which we always assume.

For any $\eta \in \cQ_{mot}$, $\eta(\ell)= m^*$.
Hence,
$$
\eta \in \hat \cQ_{mot} \quad
\Leftrightarrow 
\quad
m^*\eta \in \cQ_{mot}.
$$
In particular, 
$$
\dP_{mot}(\cdot) = m^* \ \dP(\cdot;\hat \cQ_{mot}), \quad
{\mbox{where}} \quad
\dP_{mot}(\fa):= \sup_{\eta \in \cQ_{mot}} \ \fa(\eta).
$$
 
Set
 $$
 \fH_{mot} =\left(\cH_{mot}^\perp\right)^\perp, \quad
 \fK_{mot}:= \fK_{\cQ_{mot}} =  \fK_{\hat \cQ_{mot}} .
$$
By Theorem \ref{t.main},
$$
\dP(\fa;\hat \cQ_{mot})
=\dP(\fa; \fK_{mot}),
\qquad \forall \ \fa \in \cC_\ell^\pp.
$$
One may directly verify that
$\fI(\ell-m^*) \in \fH_{mot}$.
Therefore, for any $\fa \in \cC_\ell^\pp$,
\begin{align*}
\dP_{mot}(\fa) &= m^* \ \dP(\fa;\hat \cQ_{mot}) 
= m^*\ \dD(\fa; \fK_{mot}),\\
&=m^*\  \min \left\{ c \in \R\ :
 \exists \ \fz \in \fK_{mot} \quad
 {\mbox{such that}}\quad
 c \fI(\ell) + \fz = \fa\ \right\}\\
 &= \min \left\{ c\ m^* \in \R\ :
 \exists \ \fz \in \fK_{mot} \quad
 {\mbox{such that}}\right. \\
 & \hspace{118pt} \left.
  c\ m^*{\bf{1}} +[c\ \fI(\ell-m^*) + \fz] = \fa \ \right\}.
\end{align*}
For any $\eta \in \cQ_{mot}$
and for any $\fz \in \cX^\pp$, $c \in \R$, we have,
$$
\eta(c\ \fI(\ell-m^*) + \fz)
= c [\eta(\ell)-m^*] + \eta(\fz)= \eta(\fz).
$$
Hence, $\tilde \fz := c\ \fI(\ell-m^*) + \fz \in \fK_{mot}$
if and only if $\fz \in \fK_{mot}$.  This implies that
\begin{align}
\label{e.motdual}
\dP_{mot}(\fa)
 &= \min \left\{ \tilde c \in \R\ :
 \exists \ \tilde \fz \in \fK_{mot} \quad
 {\mbox{such that}}\quad
 \tilde c {\bf{1}}+\tilde \fz = \fa\  \right\}\\
 \nonumber
 &=: \dD_{mot}(\fa).
\end{align}
We summarize the above result in the following.
\begin{thm}
\label{t.motduality}
Assume that $\mu$ and $\nu$ are 
in convex order.  Then, for every $\fa \in \cC_\ell^\pp$,
$$
\dP_{mot}(\fa) = \sup_{\eta \in \cQ_{mot}} \fa(\eta)
= \dD_{mot}(\fa).
$$
\end{thm}

\vspace{2pt}

\subsection{Map $T$}
\label{ss.mapT}

In this subsection, we write the map $T$ and its adjoints
in the coordinate form to better explain the constructions
that will be given in the next subsection.
For $i =1, \ldots, d$,  $\gamma^{(i)} \in C_b(X)$, 
$\om=(x,y)$, $x=(x^{(1)},\ldots,x^{(d)})$ and 
$y=(y^{(1)},\ldots,y^{(d)})$, set
$$
T^{(i)}(\gamma^{(i)})(\om):= \gamma^{(i)}(x) (x^{(i)}-y^{(i)}).
$$
It is clear that for $\gamma =(\gamma^{(1)},\ldots,\gamma^{(d)})
\in C_b(X;\R^d)$,
$$
T(\gamma)(\om)= \sum_{i=1}^d T^{(i)}(\gamma^{(i)})(\om).
$$
Then, $T^\p(\eta)= \left(T^{\p (1)}(\eta),\dots, T^{\p (d)}(\eta)\right)$ and
$$
T^{\p (i)}(\eta)(\gamma)=T^{\p (i)}(\eta)(\gamma^{(i)})
=\eta\left(T^{(i)}(\gamma^{(i)})\right)
=\int_\Om \gamma^{(i)}(x) (x^{(i)}-y^{(i)})\ \eta(dx,dy).
$$
For any $\fg = (\fg^{(1)}, \ldots, \fg^{(d)}) \in C_b^\pp(\Om;\R^d)$,
we have $\fg^{(i)} \in C_b^\pp(X)$ for each $i=1,\ldots,d$ and
for $\rho =(\rho^{(1)},\dots, \rho^{(d)})$ with $\rho^{(i)} \in C_b^\p(X)$,
$$
\fg(\rho)= \sum_{i=1}^d \fg^{(i)}(\rho^{(i)}).
$$
Then, for any $\eta \in \cC_\ell^\pp$,
$$
T^\pp(\fg)(\eta)= \fg(T^\p(\eta)) = \sum_{i=1}^d \fg^{(i)}(T^{\p (i)}(\eta)).
$$

\subsection{Dual Elements}
\label{ss.dualmot}

The following result characterizes the dual
elements $\fH_{mot}$.
We introduce the following sets,
  \begin{align*}
\fB_\ell&=\fB_{\ell,\mu}:= \left\{ \fa \in \cC_\ell^\pp\ :\ 
\exists \fb \in \cC_X^\pp, \ \ 
 {\mbox{such that}}\ \
 \fa=\fb \oplus {\bf{0}}\ \ {\mbox{and}}\ \ 
 \fb(\mu)=0\ \right\}\\
\fC_\ell&=\fC_{\ell,\nu}:= \left\{ \fa \in 
\cC_\ell^\pp\ :\ \exists \fc \in \cC_Y^{\prime\prime}, \ \ 
 {\mbox{such that}}\ \
 \fa={\bf{0}} \oplus \fc\ \ {\mbox{and}}\ \ 
 \fc(\nu)=0\ \right\}.
 \end{align*}

 \begin{thm}
 \label{t.mot}  Let $X=Y= \R^d$.
 Suppose that $\mu$ and $\nu$ are in convex order.
 Then,
\begin{equation}
\label{e.sum}
 \fH_{mot}=\left(\cH_{mot}^\perp\right)^\perp =  \fB_\ell + \fC_\ell + \fM.
 \end{equation}
 In particular, the dual set
 $\fK_{mot}$ is the 
 weak$^*$ closure of 
 $\fB_\ell + \fC_\ell + \fM +\cX^\pp_-$.
  \end{thm}
 \begin{proof}
It is clear from their definitions that $\fB_\ell, 
\fC_\ell$ and $\fM$ are all regularly convex. 

Step 1: We first show that $\fH:=\fB_\ell + \fC_\ell + \fM $ is regularly convex.
In view of Lemma \ref{l.condition}, regular convexity of this sum
would follow from the following estimate,
\begin{equation}
\label{e.est}
\fH \cap B^\pp_1
\subset \left(\fB_\ell  \cap B^\pp_{c^*}\right)
+  \left(\fC_\ell  \cap B^\pp_{c^*}\right) + \fM,
\end{equation}
for some constant $c^*$.

We continue by proving this estimate.
Fix $\fa_0 \in \fH \cap B_1^\pp$.  Then,
there are $\fb_0 \in \fB_\ell$, $\fc_0 \in \fC_\ell$ and
$\fg_0 \in \cC_b^\pp(X; \R^d)$ so that
$$
\fa_0= \fb_0\oplus \fc_0 +T^\pp(\fg_0).
$$
Define $\fg_1 \in \cC_b^\pp(X; \R^d)$ by,
$$
\fg_1(\rho):= \fg_0(\rho)- \fg_0(  \rho(X)\delta_{0}), \quad
\rho \in \cC_b^\prime(X; \R^d).
$$
In the above, note that for $\rho \in C_b^\p(X;\R^d)$, $\rho(A) \in \R^d$
for any Borel subset $A$ of $X$.
In the coordinate form, 
$$
\fg_0\left(\rho(X) \delta_0\right) 
=\sum_{i=1}^d \ \rho^{(i)}(X)\ \fg_0^{(i)}\left( \delta_0\right).
$$
Then, 
for any $\eta \in \cC_\ell^\p$,
$$
\fg_0\left(T^\p(\eta)(X) \delta_0\right) 
=\sum_{i=1}^d T^{\p(i)}(\eta)(X)\ \fg_0^{(i)}\left( \delta_0\right)
=\fg_0(\delta_{0}) \cdot  \int_\Om (x-y) \eta(dx, dy).
$$

Observe that
\begin{align*}
T^\pp(\fg_1)(\eta)&= \fg_1\left(T^\p(\eta)\right)
= \fg_0\left(T^\p(\eta)\right) - \fg_0\left( T^\p(\eta)(X)\delta_{0}\right)\\
&= T^\pp(\fg_0)(\eta)- 
\fg_0(\delta_{0}) \cdot  \int_\Om (x-y) \eta(dx, dy)\\
& =T^\pp(\fg_0)(\eta) + \tilde \fb_1(\eta_x) + \tilde \fc_1(\eta_y),
\end{align*}
where
$$
\tilde \fb_1(\alpha):= - \fg_0(\delta_{0}) \cdot \int_X x\ \alpha(dx),
 \quad \alpha \in \cC_X^\p,\quad
 \tilde \fc_1(\rho):= \fg_0(\delta_{0}) \cdot  \int_Y  y \ \rho(dy),
 \quad \rho \in \cC_Y^\p.
 $$
 Set
 $\fb_1:= \fb_0 -\tilde \fb_1$,
 $\fc_1:= \fc_0 -\tilde \fc_1$.
 Then,
 $\fa= \fb_1\oplus \fc_1 +T^\pp(\fg_1)$ and $\fg_1(c\ \delta_0)=0$
 for any constant $c \in \R^d$.  Since $\mu$ and $\nu$ are in convex order,
 $$
\int_X x\ \mu(dx) =\int_Y y\ \nu(dy).
$$
Therefore, $(\fb_1 \oplus \fc_1) (\mu \times \nu)=
(\fb_0 \oplus \fc_0 )(\mu \times \nu)=0$.

For  any $\beta \in \cC_Y^\p$, set $\eta_\beta:= \delta_0 \times \beta$.
We directly calculate that for $\gamma \in \cC_b(X;\R^d)$,
$$
T^\p(\eta_\beta)(\gamma) =
- \gamma(0) \cdot \int_Y y \ \beta(dy) =: \delta_0(\gamma)\cdot
c_\beta \ 
\quad
\Rightarrow
\quad
T^{\p (i)}(\eta_\beta)=c_\beta^{(i)} \delta_0, \ \ i=1,\ldots,d.
$$
Hence,
$$
T^\pp(\fg_1)(\eta_\beta)
= \fg_1\left(T^\p(\eta_\beta)\right)
=\fg_1(c_\beta \delta_0) =0.
$$
Set
$$
\fb_2(\alpha):= \fb_1(\alpha) - \fb_1(\delta_0) \alpha(X)\ \ \
\fc_2(\beta):= \fc_1(\beta) + \fb_1(\delta_0) \beta(Y).
$$
Then, since $\eta_x(X)=\eta_y(Y)= \eta(X\times Y)$,
for $\eta \in \cC_\ell^\p$,
\begin{align*}
(\fb_2\oplus \fc_2)(\eta)&=
\fb_2(\eta_x) + \fc_2(\eta_y)
=\fb_1(\eta_x) + \fc_1(\eta_y)
- \fb_1(\delta_0) [ \eta_x(X)-\eta_y(Y)]\\
&=\fb_1(\eta_x) + \fc_1(\eta_y).
\end{align*}
Therefore, the triplet $(\fb_2,\fc_2, \fg_1)$ satisfies,
$$
\fb_2\oplus \fc_2+T^\pp(\fg_1)
=\fb_1\oplus \fc_1+T^\pp(\fg_1)
=\fa
$$
and $\fb_2(\delta_0)=\fg_1(c \delta_0)=0$
for any $c\in \R^d$.  In particular, 
$$
\fc_2(\beta)= \fa(\eta_\beta), \quad
\forall \beta \in \cC_Y^\p.
$$
Hence, $\|\fc_2\|_{\cC_Y^\pp} \le \|\fa\|_{\cX^\pp} =1$.
Let $\cQ(\alpha)$ be the set of all martingale measures
$\eta$ with $\eta_x=\alpha$.  Then, for any $\eta \in \cQ(\alpha)$,
$\fb_2(\alpha) = \fa(\eta) - \fc_2(\eta_y)$.
Hence, $\|\fb_2\|_{\cC_Y^\pp} \le 2$.
Moreover,
$$
(\fb_2\oplus \fc_2)(\mu \times \nu)
=(\fb_1\oplus \fc_1)(\mu \times \nu)
=0.
$$

Set $\fb_3(\alpha):= \fb_2(\alpha)- \fb_2(\mu) \alpha(X)$,
$\fc_3(\beta):= \fc_2(\beta)- \fc_2(\nu) \beta(Y)$.
By their definitions, $\fb_3(\mu)=\fc_3(\nu)=0$.
We use these to conclude that for 
any $\eta \in \cC_\ell^\p$,
$$
(\fb_3\oplus \fc_3)(\eta)
= (\fb_2\oplus \fc_2)(\eta) - (\fb_2\oplus \fc_2)(\mu \times \nu)\ \eta(\Om)
=(\fb_2\oplus \fc_2)(\eta).
$$
The above equations imply that
$$
\fa= \fb_3\oplus \fc_3+ T^\pp(\fg_1),
$$
and $\fb_3 \in \fB_{\ell,\mu}$, $\fc_3 \in \fC_{\ell,\nu}$.  Moreover,
$$
\| \fb_3\|_{\cC_X^\pp}\le 4, \quad
\| \fc_3\|_{\cC_Y^\pp} \le 2.
$$
Therefore \reff{e.est} holds with $c^*=4$ 
and $\fB_\ell + \fC_\ell + \fM $ is regularly convex.

Step 2:
In this step, we show that the regular convexity of 
$\fB_\ell + \fC_\ell + \fM $ implies that it is equal to $ \fH_{mot}$. 
It is clear that
$$
\fB_\ell + \fC_\ell + \fM \subset \fH_{mot}.
$$
Towards a contradiction assume that
the above inclusion is strict.
By regular convexity,
there exist $\fa_0\in  \fH_{mot}$ and $\eta_0 \in \cC_\ell^\p$ such that 
\begin{align}
\label{proof.t.mot.contradiction}
\sup_{\fa\in \fB_\ell + \fC_\ell + \fM } \fa(\eta_0) <\fa_0(\eta_0).
\end{align}
By linearity, the left hand side of the above inequality
is equal to zero and therefore,
$\eta_0 \in (\fB_\ell + \fC_\ell + \fM)_\perp$.
Since $ (\fB_\ell + \fC_\ell + \fM)_\perp \subset \cH_{mot}^\perp$,
we conclude that $\eta_0 \in \cH_{mot}^\perp$.
However, $\fa_0 \in \fH_{mot}$ and consequently,
$\fa_0(\eta_0)=0$.
This is a contradiction with \eqref{proof.t.mot.contradiction}.
Consequently,
\reff{e.sum} holds and the proof of this Theorem is complete.
\end{proof}
\vspace{2pt}

\subsection{Polar Sets}

Let $\cN_{mot}$ be the set of all
 Borel subsets $Z$ of $\Om$ such that
 $\eta(Z)=0$ for every $\eta \in \cQ_{mot}$.
 It is immediate that $\fI(\chi_Z) \in \fK_{mot}$ for every
 $Z \in \cN_{mot}$. 
 However,  it is not clear whether $\fI(\chi_Z)$
 belongs to the set $\fB_\ell +
 \fC_\ell + \fM$.  Hence, these sets must be used
 in the hedging set as observed in \cite{BNT}.
  
 On the other hand, functions of the type $\chi_A$
 are not included in the original dual set $\cH_{mot}$.
 This observation suggests that the duality with the set
 $\fI(\cH_{mot})+\cX^\pp$ is not expected.
 Indeed, Example 8.1 of \cite{BNT}, a similar
 counter-example in $\cB_b(\Om)$ is
 constructed.  This example shows a duality
 gap in $\cB_b(\Om)$, when the dual elements
 do not contain the functions of the form $\chi_A$
 with $A \in \cN_{mot}$.
 
 So one needs to augment
 the set of dual elements by adding at least
 the polar sets, $\cN_{mot}$, of the
 set of probability measures $\cQ_{mot}$.
 Equivalently, one needs to consider
 all equalities and inequalities
 $\cQ_{mot}$- {\em{quasi-surely}}; 
 c.f. \cite{STZ,STZ1}.

 We close this section
 by providing constructions of some 
 polar sets discussed above
 in two separate examples.  
 In these two examples, we restrict ourselves to the 
one-dimensional case $X=Y=\R$.
 
 \begin{exm}
{\rm{   
 Let $\mu$, $\nu$ be absolutely continuous 
 with respect to the Lebesgue measure.
Consider their potential functions defined by
$$
 u_\mu (x)= \int_\R |x-t|\mu(dt),\qquad
 u_\nu (x)= \int_\R |x-t|\nu(dt).
$$
Assume that there exists $x_0\in \R$
so that
$u_\mu(x)<u_\nu (x)$ for all $x\neq x_0$ and $u_\mu(x_0)=u_{\nu}(x_0)$.
 Then, the set $A= (-\infty ,x_0) \times (x_0,\infty)$ is in $\cN_{mot}$. 
Set
$$
 a(x,y)= |y-x_0|-|x-x_0|-\frac{x-x_0}{|x-x_0|}(y-x),
 $$
 when $x\neq x_0$ and we set
 $$
 a(x,y)= |y-x_0|,\mbox{ if }x= x_0.
 $$
Also let $\fa:=\fI(a)$.
Note that by the convexity of the absolute value
function, $\fa\geq 0$ and $\fa>0$ on $A$. 
Thus, for all $n\geq 0$ 
there exists  $\fn_n \leq 0 $ 
such that $n\fa+\fn_n \in \fH_{mot}+\cX^\pp_- \subset \fK_{mot}$ 
converges to $\fI(\chi_A)$ increasingly. Therefore, by the
monotone convergence theorem,
this convergence is also in the weak$^*$ topology.
Consequently,
$\fI(\chi_A)\in \overline {\fH_{mot}+\cX^\pp_-}^*=\fK_{mot}$. 
\qed}} 
\end{exm}
\vspace{3pt}
 
 \begin{exm}
 \label{ex.gap}
{\rm{We now explain the duality gap in the Example 8.1 of \cite{BNT}
through the polar sets. The existence of the duality gap can be 
proved using the set of all 
elements in $ba(\Om)$ that are martingales
and have marginals $\mu$ and $\nu$.

In this example,
we let $\mu=\nu=\lambda$ where $\lambda$ is the Lebesgue
measure on $[0,1]$. 
Then, the only martingale coupling $\Q^*$
is the uniform probability distribution on the diagonal
$$
D:=\{(x,x): x\in[0,1]\}.
$$
Thus $D^c$, the complement of $D$, 
is a polar set of the set of measures $\cQ = \{\Q^*\}$. 

For $k\geq2$, let
 $\eta^k$ be the uniform probability measure on the set 
$$
D_k=\left\{(x,x+\frac{1}{k}): x\in [0,1-\frac{1}{k}]\right\}.
$$
Since
$D^c$ is a polar set, we have $\fI (\chi_{D^c})\in \fK_{mot}$.
On the other hand, we claim that
\begin{align}
\label{e.duality.gap}\fI (\chi_{D^c})
\notin \overline {\fB_{\ell,\lambda}+\fC_{\ell,\lambda}+\fM +\cX^\pp_-}.
\end{align}
In view  Theorem \ref{t.main}, this would 
 prove the existence of the duality gap when 
 one uses $\fB_{\ell,\lambda}+\fC_{\ell,\lambda}+\fM +\cX^\pp_-$ 
 as the hedging strategies. 

We prove \eqref{e.duality.gap}
by showing that for all
$(\fb,\fc,\fg,\fn)\in \fB_{\ell,\lambda}
\times\fC_{\ell,\lambda}\times C_b^{\prime\prime}
(\R;\R)\times \cX^\pp_- $  one has
the following estimate,
$$
\limsup_{k\to\infty} \fb(\eta^k_x)+\fc(\eta^k_y)
+T^{\prime\prime}(\fg)(\eta^k)+\fn(\eta^k) -\eta^k(D^c)\leq-1
$$
Note that $\eta^k_x$ and $\eta^k_y$ 
converges to $\lambda$ in total variation. 
Also,
$$
\sup_{|g|\leq 1} T^\prime(\eta^k)(g)
\leq  \sup_{|g|\leq 1}\left|\int g(x)(y-x) \eta^k(dx,dy)\right| 
\leq \int |y-x| \eta^k (dx,dy) \leq \frac{c_k}{k}
$$
where $c_k$ is a sequence of positive bounded constants. 
Therefore,
$$
\lim_{k\to\infty }T^{\prime\prime}(\fg)(\eta^k)
= \lim_{k\to\infty}\fg(T^\prime(\eta^k))=0.
$$
These imply that
$$
\limsup_{k\to\infty} 
\fb(\eta^k_x)+\fc(\eta^k_y)+T^{\prime\prime}(\fg)(\eta^k)
+\fn(\eta^k) -\eta^k(D^c)\leq
\limsup_{k\to\infty}   -\eta^k(D^c) = -1.
$$
Hence,  \reff{e.duality.gap} holds
and imply that the strong and the weak$^*$ closures of 
the set 
$$
\fB_{\ell,\lambda}+\fC_{\ell,\lambda}+\fM +\cX^\pp_-
$$
are distinct.  In view of the main duality theorem and the
characterization of $
\fK_{mot}$,
we conclude that there would be  a duality gap
if one uses only the above set in the definition
of the dual problem. 
\qed}}
\end{exm}
 \vspace{2pt}
 
\section{Convex envelopes}

\label{s.convex}

Assume that $X=Y$ are closed convex subsets of $\R^d$.

Motivated by the results of the previous sections,
we define the notion convexity in the bidual $\cC^{\prime\prime}$.
Recall that $ca_{r,c}(\Om)$ is the set of all countably additive Borel measures
that are compactly supported
and $ca_{r,c}^+(\Om)$ are its positive elements.  
Set
$$
\cM_{+}:= \cM \cap \left(\cC_\ell^\p\right)_+,
\quad
{\mbox{and}}
\quad
\cM_{c,+}:= \cM \cap \cC^\p_+ = \cM \cap ca_{r,c}^+(\Om).
$$

\begin{dfn}
{\rm{ We call $\fc \in \cC^{\pp}(X)$}} convex {\rm{if}}
$$
( \fc \oplus (-\fc)) (\eta) \le 0, \quad \forall
\ \eta \in \cM_{c,+} .
$$
\end{dfn}
\vspace{4pt}

To obtain equivalent characterizations of 
convexity, we recall that
$\alpha \in \cM_1(X)$ is 
in {\em{convex order}} with 
$\beta \in \cM_1(X)$ and write
$\alpha \le_c \beta$ if and only
if 
$$
\int_{X} \ \phi d\alpha \ \le \
\int_{X} \ \phi d\beta,
$$
for every $\phi:X \mapsto \R$ convex.
The following follows directly from the results 
of Strassen \cite{S}.

\begin{prp}
\label{p.convex}
$\fb \in \cC^\pp_\ell(X)$ is convex if and only if
$\fb(\alpha) \le \fb(\beta)$ for all measures 
$\alpha, \beta \in \cM_1(X)$
that are in convex order.  Moreover, if $X=\R^d$ and
if $\fb \in \cC_\ell^{\prime\prime}(\R^d)$ is convex,
then there exists $\fg \in C_b^{\prime\prime}(\R^d;\R^d)$ 
such that 
$$
\fb(\eta_x) \le \fb(\eta_y)+ T^{\prime\prime}(\fg)(\eta), \quad
\forall \ \eta \in (\cC_\ell^\prime(\Om))_+.
$$
\end{prp}
\begin{proof}
The first statement follows directly from
Strassen \cite{S}. Indeed, if a measure $\eta \in \cM_{c,+}$ then 
$\eta_x\in ca_{c,r}^+(X)$,
 $\eta_y \in ca_{c,r}^+(Y)$ and they
 are in convex order.
Conversely,
if $\alpha, \beta \in \cM_1(X) \cap ca_{c,r}^+(X)$ are in convex order,
then there exists $\eta \in  \cM_{c,+}$ such that $\eta_x=\alpha$
and $\eta_y=\beta$.  Then, the general statement follows from 
the density of $ca_{c,r}(X)$ in $\cC_\ell^\p(X)$ .

Now suppose that $X=Y=\R^d$.
Given $\fb \in \cC_\ell^{\prime\prime}(\R^d)$, define
$$
\fa:= \fb \oplus (-\fb).
$$
Then, by the definition of convexity
and Proposition \ref{p.drc},
$\fb$ is convex if and only if $\fa \in \fD + (\cC_\ell^{\prime\prime})_-$.
Hence, there exists $\fg \in C_b^{\prime\prime}(\R^d;\R^d)$
such that for any $\eta \in \cM_+$, 
$$
\fb(\eta_x)-\fb(\eta_y)=\fa(\eta) \le T^{\prime\prime}(\fg)(\eta).
$$

\end{proof}

The following is a natural extension of the classical definition.
Although a definition in the larger class $\cC^\pp$ can be given,
we restrict ourselves to $\cC_\ell(X)$ to simplify the presentation.

\begin{dfn}
\label{d.envelope}
{\rm{For any  $\fb \in \cC_\ell^{\prime\prime}(X)$ 
its}} convex envelope {\rm{is defined by,}}
$$
\fb^{c}(\alpha):= \inf \left\{ \fb(\eta_y)\ :\
\eta \in \cM_+, \ \ \eta_x=\alpha \ \right\},
\quad \alpha \in \left(\cC_\ell^\p(X)\right)_+,
$$
and for general $\alpha \in \cC_\ell^\p(X)$, 
$\fb^{c}(\alpha):= \fb^{c}(\alpha^+)
- \fb^{c}(\alpha^-)$.
\end{dfn}
\vspace{4pt}

\begin{lem}
\label{l.envelope}
For any $\fb \in \cC_\ell(X)$, $\fb^c \in \cC_\ell(X)$.
Moreover, $\| \fb \|_{\cC_\ell^\pp(X)} =
\| \fb^c \|_{\cC_\ell^\pp(X)}$ and
for every $\alpha \ge 0$
$$
\fb^{c}(\alpha) = \sup \left\{ \fc(\alpha)\ :\
\fc\ {\mbox{is convex and}}\ \ \fc \le \fb\ \right\}.
$$
\end{lem}

\begin{proof}
For $\alpha \in \left(\cC_\ell^\p(X)\right)_+$ set
$$
\cQ(\alpha):= \left\{ \eta \in \cM_+\ :\ \eta_x= \alpha\right\}.
$$
We claim that for any $\alpha, \beta \in \left(\cC_\ell^\p(X)\right)_+$,
\begin{equation}
\label{e.claim}
\cQ(\alpha)+ \cQ(\beta) = \cQ(\delta),
\quad
{\mbox{where}}
\quad
\delta :=\alpha+ \beta.
\end{equation}
Indeed, the inclusion
$\cQ(\alpha)+ \cQ(\beta) \subset  \cQ(\delta)$
follows directly from the definitions.    For
any Borel subset $A \subset X$, set
$$
z^\alpha:= \frac{d \alpha}{d \delta}, \quad
z^\beta:= \frac{d \beta}{d \delta},
\quad \delta=\alpha+\beta.
$$
Then, both $z^\alpha$, and $z^\beta$ are Borel functions of the
first variable $x$.  Moreover,  
$z^\alpha(x), z^\beta(x) \in [0,1]$ and $z^\alpha+z^\beta \equiv 1$.
For any $\eta \in \cQ(\delta)$, set
$$
\eta^\alpha(A):= \int_A \ z^\alpha(x) \eta(dx, dy), 
\quad
\eta^\beta(A):= \int_A \ z^\beta(x) \eta(dx, dy).
$$
It is clear that $\eta^\alpha, \eta^\beta \in (\cC^\p_\ell)_+$.
For any $h \in \cC_\ell(X)$ and 
$\eta \in \cQ(\delta)$, we calculate that
\begin{align*}
\int_X h(x) \eta^\alpha_x(dx)&=
\int_\Om h(x) \eta^\alpha(dx,dy)=
\int_\Om h(x) z^\alpha(x) \eta(dx,dy)\\
&=
\int_X h(x) z^\alpha(x) \eta_x(dx) =
\int_X h(x) z^\alpha(x) \delta(dx) \\
&= \int_X h(x)  \alpha(dx).
\end{align*} 
So, we conclude that $\eta^\alpha_x=\alpha$.
Also for any $\gamma \in \cC_b(X;\R^d)$,
$$
\int_\Om T(\gamma) d\eta^\alpha= 
\int_X \ z^\alpha(x) \gamma(x) \cdot
\left( \int_Y (x-y) \eta(dx,dy)\right)=0.
$$
Hence, $\eta^\alpha \in \cQ(\alpha)$.
Similarly one can show that $\eta^\beta \in \cQ(\beta)$.
Since $\eta^\alpha + \eta^\beta = \eta$, this proves 
the claim \reff{e.claim}.

The linearity of the map $\alpha \in \cC^\p_\ell(X)$ to
$\fb^c(\alpha)$ now follows directly from the definitions and \reff{e.claim}.
Similarly the norm statement is immediate from the definitions.
\end{proof}
\vspace{2pt}

 \section{Appendix: Regularly Convex Sets} 
 \label{ap.rc}
 
 In this Appendix we state a slight
 extension of a condition for
 regular convexity that is proved in \cite{KS}. 
 Let $E$ be any Banach space.
 First note that a set $\fK \in E^\p$
 is  regularly convex if and only if it
 is closed in the weak$^*$ topology and is convex.
 The following result is an immediate 
 corollary of Theorem 7 of \cite{KS}
 and is used repeatedly in our arguments.
 
 \begin{lem}
 \label{l.condition}
Let $\fK \subset E^\prime$.
Suppose that for each $R>0$ 
there exists a regularly convex
set $\fL_{R}$ so that
$$
\fK \cap B_{R} \subset \fL_{R} \subset \fK,
$$
where $B_{R}$ is the closed ball in $E^\prime$
centered around the origin with radius $R$.  Then, $\fK$ is
regularly convex. 
\end{lem}
\begin{proof}
Let $\fU$ be bounded, regularly convex and
$B_{R}$ be a ball that contains $\fU$.
Then, we have the set inclusions
$$
\fK \cap \fU \subset \fK \cap B_R \subset
\fL_R \subset \fK.
$$
Therefore,
$$
\fK \cap \fU = \fL_R \cap \fU.
$$
Since both $\fL_R$ and $\fU$ are regularly convex,
so is their intersection.  Therefore,
for every regularly convex, bounded $\fU$,
$\fK \cap \fU$ is regularly convex.  
By Theorem 7 \cite{KS},
this proves the regular convexity of $\fK$.
 \end{proof}
 \vspace{2pt}
 
 \noindent
 {\bf{Acknowledgements.}}
 
 The authors would like to thank Matti Kiiski for 
  many comments and fruitful discussions
  and  would like to thank the
anonymous reviewer for valuable
comments.

\end{document}